\documentclass[a4paper, 10pt, parskip=half]{scrartcl}

\usepackage[utf8]{inputenc}
\usepackage[T1]{fontenc}
\usepackage{lmodern}
\usepackage{amsmath}
\usepackage{amssymb}
\usepackage{amsthm}
\usepackage{amsfonts}
\usepackage{dsfont}
\usepackage{mathtools}
\usepackage{hyperref}
\hypersetup{
  colorlinks=true,
  linkcolor=black,
  citecolor=black,
  urlcolor=blue,
  pdftitle={Unboundedness phenomenon in a model of urban crime},
  pdfauthor={Mario Fuest, Frederic Heihoff},
  pdfkeywords={chemotaxis, urban crime, singular limit, logarithmic sensitivity},
  bookmarksopen=true,
}

\usepackage[numbers, sort&compress]{natbib}

\RequirePackage{geometry}
\newcommand{\R}{\mathbb{R}}

\newcommand{\N}{\mathbb{N}}

\newcommand{\ur}[1]{\mathrm{#1}}
\newcommand{\ure}{\ur e}

\ifdefined\labelenumi
  \renewcommand{\labelenumi}{(\roman{enumi})}
  
\fi

\newcommand{\eps}{\varepsilon}

\newcommand{\gt}{>}
\newcommand{\lt}{<}

\newcommand{\defs}{\coloneqq}
\newcommand{\sfed}{\eqqcolon}

\newcommand{\ra}{\rightarrow}

\newcommand{\nea}{\nearrow}

\newcommand{\ol}{\overline}

\newcommand{\dx}{\,\mathrm{d}x}
\newcommand{\ds}{\,\mathrm{d}s}

\newcommand{\drho}{\,\mathrm{d}\rho}

\newcommand{\ddt}{\frac{\mathrm{d}}{\mathrm{d}t}}

\newcommand{\embed}{\hookrightarrow}

\newcommand{\hp}{\hphantom}
\newcommand{\pe}{\mathrel{\hp{=}}}

\newcommand{\tmax}{T_{\max}}

\newcommand{\intom}{\int_\Omega}

\newcommand{\Ombar}{\ol \Omega}

\newcommand{\leb}[2][\Omega]{\ensuremath{L^{#2}(#1)}}

\newcommand{\sob}[3][\Omega]{\ensuremath{W^{#2, #3}(#1)}}

\newcommand{\con}[2][\Ombar]{\ensuremath{C^{#2}(#1)}}

\newcommand{\grad}{\nabla}
\newcommand{\laplace}{\Delta}

\newcommand{\ue}{u_\eps}
\newcommand{\ve}{v_\eps}

\newcommand{\tmaxeps}{T_{\max, \eps}}
\renewcommand{\L}[1]{{L^{#1}(\Omega)}}

\let\originalparagraph\paragraph
\renewcommand{\paragraph}[2][.]{\originalparagraph{#2#1}}

\makeatletter
\renewenvironment{proof}[1][\proofname]{\par
  \pushQED{\qed}%
  \normalfont \topsep0\p@\relax
  \trivlist
  \item[\hskip\labelsep\scshape
  #1\@addpunct{.}]\ignorespaces
}{%
  \popQED\endtrivlist\@endpefalse
}
\makeatother

\newtheorem{base}{Base}[section]
\numberwithin{equation}{section}

\newtheorem{theorem}[base]{Theorem} \newtheorem*{theorem*}{Theroem}
\newtheorem{lemma}[base]{Lemma} \newtheorem*{lemma*}{Lemma}
 \newtheorem*{prop*}{Proposition}
 \newtheorem*{cor*}{Corollary}

 \newtheorem*{definition*}{Definition}
 \newtheorem*{example*}{Example}
 \newtheorem*{cond*}{Condition}

\newtheorem{remark}[base]{Remark} \newtheorem*{remark*}{Remark}

\setkomafont{title}{\normalfont \Large}
\title{Unboundedness phenomenon in a model of urban crime}

\usepackage{authblk}

\author[1,2]{Mario Fuest\footnote{e-mail: fuest@ifam.uni-hannover.de, corresponding author}}
\author[2]{Frederic Heihoff\footnote{e-mail: fheihoff@math.upb.de}}
\affil[1]{Leibniz Universität Hannover, Institut für Angewandte Mathematik, Welfengarten 1, 30167 Hannover, Germany}
\affil[2]{Institut für Mathematik, Universität Paderborn, Warburger Str.~100, 33098 Paderborn, Germany}  
\date{}

\begin{document}
\maketitle

\KOMAoptions{abstract=true}
\begin{abstract}
  \noindent
  We show that spatial patterns (`hotspots') may form in the crime model
	\begin{equation}\label{prob:star}\tag{$\star$}
	\left\{\;
	\begin{aligned}
	u_{t} &= \tfrac{1}{\varepsilon}\Delta u - \tfrac{\chi}{\varepsilon} \nabla \cdot \left(\tfrac{u}{v} \nabla v \right) - \varepsilon uv, \\
	v_{t} &= \Delta v - v + u v,
	\end{aligned}
	\right.
	\end{equation}
  which we consider in $\Omega = B_R(0) \subset \mathbb R^n$, $R > 0$, $n \geq 3$ with $\eps > 0$, $\chi > 0$ and initial data $u_0$, $v_0$ with sufficiently large initial mass $m := \int_\Omega u_0$.
  More precisely, for each $T > 0$ and fixed $\Omega$, $\chi$ and (large) $m$, we construct initial data $v_0$ exhibiting the following unboundedness phenomenon: Given any $M>0$, we can find $\varepsilon > 0$ such that the first component of the associated maximal solution becomes larger than $M$ at some point in $\Omega$ before the time $T$. Since the $L^1$ norm of $u$ is decreasing, this implies that some heterogeneous structure must form.\\[0.5em]
  We do this by first constructing classical solutions to the nonlocal scalar problem 
  \[
  	w_t = \Delta w + m \frac{w^{\chi+1}}{\int_\Omega w^\chi}
  \]
  from the solutions to the crime model by taking the limit $\varepsilon \searrow 0$ under the assumption that the unboundedness phenomenon explicitly does not occur on some interval $(0,T)$. We then construct initial data for this scalar problem leading to blow-up before time $T$. As solutions to the scalar problem are unique, this proves our central result by contradiction.\\[0.5em]
\end{abstract}

{\small \textbf{Note for second version:}
In the first version of this article, we studied a reduced version of the crime model where the term $-\eps uv$ in the first equation in \eqref{prob:star} is replaced by $0$.
The present second version includes this term and thus we removed the word “reduced” from the title.
Additionally, we incorporated several helpful comments by the anonymous reviewers, for which we are thankful.}

\section{Introduction}
In the present paper, we study properties of the crime model introduced in \cite{ShortStatisticalModelCriminal2008}, which belongs to the large class of taxis-type systems.
Such cross-diffusive systems have not only proven to be a fertile ground for the modeling of biological processes involving chemotaxis since their introduction as a tool in these contexts in the highly influential paper by Keller and Segel in 1970 (cf.\ \cite{KellerInitiationSlimeMold1970}), but have also found application in other fields of study, such as the analysis of patterns in urban crime (cf.\ \cite{ShortStatisticalModelCriminal2008}). In general, said models take approximately the form
\begin{equation}\label{eq:general_problem}
	\left\{\;
	\begin{aligned}
	u_{t} &= \grad \cdot ( D(u,v) \grad u) - \grad \cdot \left(S(u,v) \grad v \right) + f(u,v), \\
	v_{t} &= \laplace v + g(u,v),
	\end{aligned}
	\right.
\end{equation}
where the first equation models the spatial density of some entities, such as cells or criminals, whose otherwise random movement is affected by what is modeled in the second equation, which can be an abstract concept like spatial attractiveness in terms of criminal activity or something more concrete such as the concentration of a chemical compound. The term representing this $v$-directed movement mechanism and thus generally the term of central interest in the above system is $\grad \cdot \left(S(u,v) \grad v \right)$, the taxis term. The source terms $f$ and $g$ represent growth, decay and potential interactions between the model components not involving any derivatives.

The most prominent representative of this model class is the classical Keller--Segel system (\eqref{eq:general_problem} with $D(u, v) = 1$, $S(u, v) = u$, $f \equiv 0$ and $g(u, v) = - v + u$),
which has been suggested to model certain aggregation behaviors of slime mold (cf.\ \cite{KellerInitiationSlimeMold1970}). A natural and important question to ask is whether this system and variants thereof allow for classical solutions blowing up in finite time---which can be interpreted as the most drastic indication of pattern formation.
For the classical Keller--Segel system, affirmative answers have been given in \cite{HerreroVelazquezBlowupMechanismChemotaxis1997} and \cite{MizoguchiWinklerBlowupTwodimensionalParabolic} for the two-dimensional setting under a largeness condition on the initial mass and in \cite{WinklerFinitetimeBlowupHigherdimensional2013} for the higher dimensional case (see also  \cite{JagerLuckhausExplosionsSolutionsSystem1992} and \cite{NagaiBlowupRadiallySymmetric1995} for early results concerning finite-time blow-up in certain parabolic--elliptic simplifications of that system).
These findings are complemented by global existence results in one dimension (cf.\ \cite{OsakiYagiFiniteDimensionalAttractor2001}) and in two dimensions under a smallness condition for the initial mass (cf.\ \cite{NagaiEtAlApplicationTrudingerMoserInequality1997}).
While the analysis of global existence and blow-up of systems of type \eqref{eq:general_problem} forms a large part of the chemotaxis literature, many more aspects of such systems have also been studied. Instead of giving an overview of these results here, we choose to refer to the surveys \cite{BellomoEtAlMathematicalTheoryKeller2015} and \cite{LankeitWinklerFacingLowRegularity2019} instead.

\paragraph{The crime model} The system of equations we are interested in is the urban crime model (without inhomogeneous production terms) introduced in \cite{ShortStatisticalModelCriminal2008}; that is, we consider
\begin{equation}\label{problem}
\left\{\;
\begin{aligned}
u_{\eps t} &= \tfrac{1}{\eps}\laplace \ue - \tfrac{\chi}{\eps}\grad \cdot \left(\tfrac{\ue}{\ve} \grad \ve \right) - \varepsilon \ue \ve && \text{in } \Omega \times(0,\infty), \\
v_{\eps t} &= \laplace \ve - \ve + \ue \ve && \text{in } \Omega \times(0,\infty), \\
\grad \ue \cdot \nu &= 0, \;\; \grad \ve \cdot \nu = 0 && \text{on } \partial \Omega \times (0,\infty), \\
\ue(\cdot, 0) &= u_0, \;\; \ve(\cdot, 0) = v_0 && \text{in } \Omega 
\end{aligned}
\right.
\end{equation}
in $\Omega \defs B_R(0) \subset \R^n$, $n \ge 3$, $R > 0$, for all $\eps > 0$ and a fixed parameter $\chi > 0$.
Here, $\ue$ denotes the density of burglars whose movement is partially modelled to be random (term $\frac1\eps \Delta \ue$) in order to account for effects of `routine activities' (cf.\ \cite{FelsonRoutineActivitiesCrime1987}, \cite{CohenFelsonSocialChangeCrime1979}); that is, the observation that criminals generally choose targets nearby their usual locations and hence act locally.  However, they also partially orient their movement towards (the logarithm of) an abstract attractiveness field $\ve$ (term $-\frac{\chi}{\eps} \nabla \cdot (\frac{\ue}{\ve} \nabla \ve)$), which diffuses in space (term $\Delta \ve)$, decays over time (term $-\ve$) and, importantly, is increased whenever a crime happens; that is, when criminals arrive at an attractive spot (term $+\ue \ve$). Thereby, effects of `repeat and near-repeat victimization' (cf.\ \cite{ShortMeasuringModelingRepeat2009}, \cite{JohnsonNewInsightsSpatial1997}, \cite{WilsonBrokenWindows1982}) are modelled; that is, areas where a crime has happened nearby become more attractive, for instance because burglars are already familiar with the neighbourhood or because a large amount of broken windows may be perceived as a lack of care for criminal activity. Finally, the term $-\eps \ue \ve$ in the first equation models that a fraction of burglars stops its activity after a successful crime. We also note that the system derived in \cite{ShortStatisticalModelCriminal2008} also includes inhomogeneous production terms for both equations (which we have set to zero here), representing that new criminal agents are added over time and that certain areas are intrinsically attractive.

The factors in front of these terms reflect the strength of the corresponding effect; we have set several of them to $1$ for simplicity.
However, as we are in particular interested in spatial aggregation phenomena of solutions to \eqref{problem} when the evolution of $\ue$ is mostly determined by movement and much less so by the withdrawal of criminals,
we introduce a (small) parameter $\eps$. Moreover, we emphasize that the strength of the taxis term $\frac{\chi}{\eps}$ may be arbitrary small compared to the strength of diffusion $\frac1\eps$.

Before stating our main result, we first discuss known results of \eqref{problem} and related systems. Global existence of classical solutions to \eqref{problem} in one spatial dimension has been shown in \cite{RodriguezGlobalExistenceQualitative2019}, while similar results have been achieved in \cite{FreitagGlobalSolutionsHigherdimensional2018} and \cite{AhnEtAlGlobalWellposednessLogarithmic2021} for higher dimensions under certain smallness assumptions. In two dimensions and with a sufficiently strong nonlinear diffusion term in the first equation, it is possible to construct weak solutions as seen in \cite{RodriguezRelaxationNonlinearDiffusion2020}. Moreover, again in the two-dimensional setting, global generalized solutions have been constructed either assuming radial symmetry (cf.\ \cite{WinklerGlobalSolvabilityStabilization2019}) or the presence of a logistic source term in the first equation (cf.\ \cite{HeihoffGeneralizedSolutionsSystem2020}). Some of the mentioned existence results further incorporate some discussion of long time behavior and stabilization as the formation of hotspots (cf.\ \cite{BerestyckiExistenceSymmetricAsymmetric2014}, \cite{CantrellGlobalBifurcationSolutions2012}, \cite{TseHotspotFormationDynamics2016}) is of significant interest from an application perspective. To our knowledge, there are no blow-up results available for the crime model yet.

A well-known relative of \eqref{problem} is the Keller--Segel system with logarithmic sensitivity,
\begin{equation}\label{eq:problem:ks_log_sens}
\left\{\;
\begin{aligned}
u_t &= \laplace u - \chi \grad \cdot \left(\tfrac{u}{v} \grad v \right), \\
v_t &= \laplace v - v + u,
\end{aligned}
\right.
\end{equation}
whose main difference compared to \eqref{problem} is that the production term in the second equation is $+u$ instead of $+uv$, which makes it easier to obtain helpful a priori bounds as long as boundedness of $v$ is not yet guaranteed. We first mention the work that inspired the present paper, namely \cite{WinklerSingularLimit2021}, in which initial data are constructed that lead to a very similar unboundedness phenomenon in that system as obtained for \eqref{problem} in Theorem~\ref{th:main} below. Apart from this, the knowledge concerning this system seems to similarly focus on global existence results:
for the one-dimensional setting (cf.\ \cite{TaoLargetimeBehaviorParabolicparabolic2013}), under certain smallness conditions (cf.\ \cite{WinklerGlobalSolutionsFully2011}, \cite{LankeitNewApproachBoundedness2016}, \cite{AhnEtAlGlobalWellposednessLogarithmic2021}), for specific parameter correlations (cf.\ \cite{ZhangGlobalBoundedSolutions2020}) or by relaxing the solution concept (cf.\ \cite{WinklerGlobalSolutionsFully2011}, \cite{StinnerWinklerGlobalWeakSolutions2011}, \cite{LankeitGeneralizedSolutionConcept2017a}). Unlike the crime model, \eqref{eq:problem:ks_log_sens} can be simplified to a parabolic--elliptic system in a straightforward manner and the knowledge regarding that system is more complete: Not only are classical solutions known to exist under weaker conditions (cf.\ \cite{FujieSenbaGlobalExistenceBoundedness2015}, \cite{NagaiSenbaGlobalExistenceBlowup1998}), for sufficiently large values of $\chi$ solutions blowing up in finite-time have been constructed in \cite{NagaiSenbaGlobalExistenceBlowup1998}.

\paragraph{Main result}
Instead of trying to obtain further global existence results, in the present paper we take a different perspective and aim to show that \eqref{problem} features certain unboundedness phenomena.
However, as noted for instance in the survey \cite{LankeitWinklerFacingLowRegularity2019},
all techniques for constructing blow-up solutions of parabolic--parabolic systems of the form \eqref{eq:general_problem} seem to heavily make use of certain energy functionals---%
which are only known to exist in some special cases.
Still, one might hope to observe less drastic versions of pattern formation in parabolic--parabolic cross-diffusion systems without relying on energy functionals.
Indeed, apart from the already mentioned result \cite{WinklerSingularLimit2021} regarding logarithmic Keller--Segel systems,
for the Keller--Segel system with logistic source it has been found in \cite{WinklerEmergenceLargePopulation2017} that population densities may surpass so-called carrying capacities
(cf.\ also \cite{WinklerHowFarCan2014}, \cite{LankeitChemotaxisCanPrevent2015} for analogous results for parabolic--elliptic variants of that system),
which additionally shows that pattern formation on intermediate timescales of this kind is possible even for global classical solutions.

We now extend these results to the crime model \eqref{problem}.
That is, the main insight of this paper is that given initial data $u_0$ with sufficiently large mass, for any time $T > 0$ it is possible to construct initial data $v_0$ such that the associated solutions to (\ref{problem}) exhibit the following unboundedness property: The local solutions to (\ref{problem}) associated with the aforementioned initial data grow arbitrarily large in any $L^p(\Omega)$, $p > \frac{n}{2}$ at times before $T > 0$ as $\eps$ becomes small.   More precisely, we prove the following
\begin{theorem}\label{th:main} 
	Let $\chi > 0$, $n \ge 3$ and $\Omega \defs B_R(0) \subset \R^n$ with some $R > 0$. 

	There exists $m_0 > 0$ such that for each $T > 0$ we can construct initial data $v_0 \in W^{1,\infty}(\Omega)$ being positive in $\Ombar$ with the following property: For all nonnegative initial data $u_0 \in C^0(\overline{\Omega})$ with $\int_\Omega u_0 > m_0$, it is possible to find $(T_\eps)_{\eps \in (0, 1)} \subset (0, T]$ and $(\ue, \ve)_{\eps \in (0, 1)} \subset C^0(\overline{\Omega}\times[0,T_\eps)) \cap C^{2,1}(\overline{\Omega}\times(0,T_\eps))$ with $\ue \geq 0$ and $\ve > 0$ in $\Ombar \times [0, T_\eps)$ for $\eps \in (0, 1)$ such that $(\ue, \ve)$ is a classical solution to (\ref{problem}) in $\Ombar\times[0,T_\eps)$ for $\eps \in (0, 1)$ and that
	\begin{equation}\label{eq:th:main}
	\limsup_{\eps \searrow 0} \sup_{t\in(0,T_\eps)} \|\ue(\cdot, t)\|_\L{p} = \infty
	\qquad \text{for all $p > \frac n2$}.
	\end{equation}
	This further implies that, for each $M > 0$, there exist $\eps \in (0,1)$, $x \in \Omega$, $t \in (0,T_\eps)$ such that $\ue(x,t) > M$.
\end{theorem}

\begin{remark}
  As the mass of the first solution component is decreasing in time, Theorem~\ref{th:main} in particular shows that spatial patterns form for certain choices of parameters and initial data; largeness in $\leb\infty$ and (relative) smallness in $\leb1$ can only be obtained by heterogeneous structures. Rigorously establishing pattern formation further supports the assumptions made in \cite{ShortStatisticalModelCriminal2008} for deriving the model \eqref{problem}, as this indicates that the system is indeed capable of displaying the observed criminal hotspots.
\end{remark}

\begin{remark}
  Let us briefly state two key differences between Theorem~\ref{th:main} and \cite[Theorem~1.1]{WinklerSingularLimit2021},
  where a similar result has been obtained for the Keller--Segel system with logarithmic sensitivity \eqref{eq:problem:ks_log_sens},
  which differs from \eqref{problem} mainly due to its weaker production term $+\ue$ instead of $+\ue\ve$ in the second equation.
  
  First, instead of $\chi > \frac{n}{n-2}$, we merely need to require positivity of $\chi$; that is, for the crime model an arbitrary weak taxis term (when compared to the diffusion term) already suffices to obtain such an unboundedness phenomenon.
  Second, for reasons discussed at the end of the introduction, we need to additionally require largeness of $\intom u_0$.
\end{remark}

\paragraph{Approach and complications} 

Our approach can be neatly separated into two steps, which each present a standalone result, and is fairly closely based on the methods presented in \cite{WinklerSingularLimit2021} for a similar system. 

For the first step, which we cover in Section~\ref{sec:reduction}, we work under the assumption that, for given initial data $u_0$, $v_0$, the local solutions $(\ue, \ve)$ of (\ref{problem}) exist on at least some time interval $(0,T)$ and are uniformly bounded regarding their first component in some $L^p(\Omega)$ space with $p > \frac{n}{2}$ independent of the parameter $\eps$. Our ultimate goal in this section is to use this boundedness assumption to construct a weak solution to a limit version of (\ref{problem}) and then combine the components of said weak solution into a classical solution to the scalar Neumann problem associated with the partial differential equation
\begin{equation}\label{eq:approach:singular}
	w_t = \laplace w + m \frac{w^\alpha}{\int_\Omega w^\chi}
\end{equation}
on $(0,T)$ with initial data $w_0 = v_0$,  $\alpha = \chi + 1$ and $m = \int_\Omega u_0$. To do this, we begin by using our assumed bound for $\ue$ as a baseline to derive further a priori estimates for both solution components by way of either testing based or semigroup based methods. We then employ the compact embedding properties of certain function spaces, e.g.\ those ensured by the Aubin--Lions lemma, to gain a null sequence $(\eps_j)_{j\in\N}$, along which the solutions $(\ue, \ve)$ converge to a pair of functions $(u,v)$ in such way as to retain sufficient weak solution properties for $u$ and $v$ to be a weak solution to the limit version of (\ref{problem}) which is obtained by multiplying the first equation in (\ref{problem}) by $\eps$ and then setting $\eps = 0$. We then largely reuse but also slightly adapt the arguments presented in \cite{WinklerSingularLimit2021} to gain an explicit form for $u$ in terms of $v$ by first proving that the function $u$ has constant mass $m = \int_\Omega u_0$ and then essentially testing the first equation in the limit version of (\ref{problem}) with $\ln(u) - \chi \grad \ln(v)$. Replacing $u$ by its explicit formula in the weak formulation for the second subproblem and using some standard parabolic regularity theory then completes this section.

Section~\ref{sec:limit_blowup} focuses exclusively on solutions to the scalar equation (\ref{eq:approach:singular}) with $\alpha = \chi + 1$. We first establish the existence of unique classical solutions with an associated blow-up criterion independent of the construction method presented in the previous section. Our goal in this section is to construct (radial) initial data $w_0$ for each $T > 0$ such that the associated solution to (\ref{eq:approach:singular}) blows up before time $T$ given that $m$ was sufficiently large. As this would contradict our construction from the previous section, its core assumption must have been wrong, which directly implies Theorem~\ref{th:main}. Because achieving this blow-up result, which is arguably of some independent interest, is in many ways the central argument of this paper, let us give a brief overview over its associated methods, inspirations and challenges:

Our approach is inspired by the arguments presented in \cite{QuittnerSoupletSuperlinearParabolicProblems2019}, which in turn build on ideas going back to \cite{FriedmanMcLeodBlowupPositiveSolutions1985} and \cite{HuYinSemilinearParabolicEquations1995}, where the case $\alpha = \chi$ is discussed and blow-up is in fact achieved for sufficiently large values of $\chi$ without any condition on $m$. Similar to our inspirations, our first step in this section is the construction of a family of initial data candidates $w_0$, which already resemble a blow-up state as they are essentially smoothed versions of the singular function $|x|^{-\frac{2}{\chi}}$. Maybe somewhat counter-intuitively, our next step is the establishment of a uniform upper bound for the solutions associated with these initial data candidates. This bound is then in fact used to ensure a lower bound for the dampening term $k(t) \defs m \left(\int_\Omega w(\cdot, t) \right)^{-1}$ in an effort to make the solutions associated with our initial data available to a comparison with Neumann problems of the form $z_t = \laplace z + 2\lambda z^{\chi+1}$, which, for sufficiently large $\lambda$, can be shown to blow up as early as desired by choosing initial data sufficiently close to $|x|^{-\frac{2}{\chi}}$. It is this point, where our need for a condition on $m$ becomes apparent, as we need $k(t)$ not only to be bounded from below but to be specifically larger than $2\lambda$ for an appropriate $\lambda$. As we only have limited control over the actual bound for $w$, this appears to only be achievable by prescribing a sufficient lower bound $m_0$ for $m$. Notably and in contrast to the case $\alpha=\chi$ already considered in the literature, rescaling of the whole solution family does not help in this case as such a rescaling would also affect the lower bound for the parameter $\lambda$ in the same way.

The derivation of the aforementioned upper bound for $w$ is now the last missing piece of the puzzle and is again based on a comparison argument. This time we consider the function $J \defs w_r + \eta r w^q$, $q \in (1,\chi+1)$, with some appropriate $\eta$, which, if shown to be nonpositive on $(0,R/2)$, gives us our desired bound as our solutions are radially decreasing. While the derivation of the necessary parabolic equation for $J$ is very similar to the one in \cite{QuittnerSoupletSuperlinearParabolicProblems2019}, a key difficulty in our scenario stems from the need for nonpositive boundary information at $R/2$ as a prerequisite for the application of the maximum principle. While in the case of $\alpha = \chi$, this is a rather straightforward consequence of local boundedness of mass, we cannot rely on any such property for our solutions. Instead we need to employ a much more subtle argument in Lemma~\ref{lm:w_pw_local_bdd} to control the value of $w(R/2, \cdot)$. Combined with the negative upper bound derived in Lemma~\ref{lm:w_r_bdd} for $w_r(R/2, \cdot)$ and after potential adjustment of $\eta$, this then allows for the application of the maximum principle to gain our desired bound. This effectively closes the argument.

As seen above, our mass condition is of course an artifact of our methods first and foremost and not necessarily an intrinsic property of the equation. As such, it is certainly an interesting future question whether solutions to (\ref{eq:approach:singular}) always globally exist for $\alpha = \chi + 1$ and small $m$ or if a similar unboundedness result without the condition on $m$ can be achieved by different means.

Interestingly though from a structural point of view, one would perhaps not expect the same mass condition for both the cases $\alpha > \chi + 1$ and $\alpha \in (\chi, \chi + 1)$ as the need for it could been seen as a consequence of the unique scaling invariance property of the case $\alpha = \chi + 1$, which robs us of one degree of freedom in our choice of parameters. In fact, the case $\alpha > \chi + 1$ is quite trivial as, for sufficiently large initial mass $\int_\Omega w_0$, the $L^1(\Omega)$ norm of the associated solution already blows up as early as we desire for any $m$. Conversely in the case $\alpha \in (\chi, \chi + 1)$, rescaling a family of solutions essentially adjusts the value of $m$ in the equation to any value desired. This means that any largeness condition for $m$ should be achievable by such a rescaling without affecting blow-up times, while the remainder of our argument should rather easily translate. In this sense, the scenario discussed here represents an interesting boundary case.

\section{Reduction to the limit problem}\label{sec:reduction}

As already laid out in the introduction, this section will feature the first half of the core argument presented in this paper. As such keeping in mind that we want to prove our central result by contradiction, we will in this section only consider initial data, for which the blow-up type scenario outlined in Theorem~\ref{th:main} does not happen, or more specifically we work under the assumption that for some fixed initial data the solutions to (\ref{problem}), or more precisely their first components, do not blow-up in $L^p(\Omega)$, $p > \frac{n}{2}$ on some fixed time interval $(0,T)$ even as $\eps \searrow 0$. The central goal of this section is to then construct a solution $w$ to the scalar Neumann problem (\ref{eq:approach:singular}) from $u$ and $v$ following the approach described in the introduction.

We thus begin by presenting the following local existence result for the system (\ref{problem}) as a necessary prerequisite for the formulation of this section's main result.
\begin{lemma}\label{lm:local_exist}
	Let $\eps \in (0,1)$, $\chi > 0$, $n\geq 3$ and $\Omega \subset \R^n$ be a bounded convex domain with a smooth boundary. Further let $u_0 \in C^0(\overline{\Omega})$ with $u_0 \geq 0$ and $v_0 \in W^{1,\infty}(\Omega)$ with $v_0 > 0$ in $\Ombar$ be some initial data. 

	Then there exist $\tmaxeps \in (0,\infty]$ and uniquely determined functions
  \begin{align*}
    \ue &\in C^0(\overline{\Omega} \times[0,\tmaxeps)) \cap C^{2,1}(\overline{\Omega} \times(0,\tmaxeps)), \\
    \ve &\in \bigcup_{q > n} C^0([0,\tmaxeps); \sob1q) \cap C^{2,1}(\overline{\Omega} \times(0,\tmaxeps))
  \end{align*}
  that solve (\ref{problem}) classically and comply with the following blow-up criterion:
	\begin{equation}\label{eq:blowup_criterion}
	\text{If } \tmaxeps < \infty, \quad \text{ then } \quad \limsup_{t\nearrow \tmaxeps} \|\ue(\cdot, t)\|_\L{p} = \infty \quad \text{ for all } p > \frac{n}{2}.
	\end{equation}
	Further,
	\begin{equation}\label{eq:mass_bound}
	\int_\Omega \ue(\cdot, t) = \int_\Omega u_0 - \varepsilon\int_0^t\int_\Omega \ue(\cdot, s) \ve(\cdot, s) \ds \leq \int_\Omega u_0
	\end{equation}
	for all $t \in (0,\tmaxeps)$.
\end{lemma} 
\begin{proof}
	This kind of local existence result is the consequence of a standard contraction mapping argument, which is already well established in the literature concerning taxis systems of the type seen in (\ref{problem}). Therefore we refer to e.g.\ \cite{BellomoEtAlMathematicalTheoryKeller2015} instead of laying out the argument in full.
  Moreover, \eqref{eq:mass_bound} directly follows upon integrating the first equation in \eqref{problem} in time and space.

	We note that the blow-up criteria for similar systems in the literature generally also deal with the second solution component $\ve$ either blowing up in $W^{1,q}(\Omega)$, $q > n$, or its 
	infimum becoming arbitrarily small. As we will see in Lemma~\ref{lm:ve_lower_bound} and Lemma~\ref{lm:ve_w1q_bound} below, blow-up of $\ve$ in this way already necessitates the blow-up characterized in (\ref{eq:blowup_criterion}) and therefore we will not devote any time to this argument here.
\end{proof}

Given this existence result, we can now formulate the central lemma of this section as follows:

\begin{lemma}\label{lm:limit_reduction}
	Let $\chi > 0$ and $n \geq 3$ and $\Omega \subset \R^n$ be a bounded convex domain with a smooth boundary. Further let $u_0 \in C^0(\overline{\Omega})$ with $u_0 \geq 0$ and $v_0 \in W^{1,\infty}(\Omega)$ with $v_0 > 0$ in $\Ombar$ be some initial data. Let then $(\ue, \ve)$ be the unique maximally extended solutions to (\ref{problem}) on $\Ombar \times [0,\tmaxeps)$ for each $\eps \in (0,1)$ as constructed in Lemma~\ref{lm:local_exist}.

	If there exist $T \in (0,\infty)$, $\eps^\star \in (0,1)$ and some $p > \frac{n}{2}$ such that
	\begin{equation}\label{eq:assumption}
	\inf_{\eps\in(0,\eps^\star)}\tmaxeps > T \quad \text{ and } \sup_{\eps \in (0,\eps^*)} \sup_{t\in(0,T)} \|\ue(\cdot,t)\|_\L{p} < \infty, \tag{A}
	\end{equation}
  \renewcommand*{\theHequation}{notag.\theequation} 
	then there exists a positive classical solution $w \in C^0(\overline{\Omega}\times[0,T)) \cap C^{2,1}(\overline{\Omega}\times(0,T))$ of the system
	\begin{equation}\label{limit_problem}
	\left\{\;
	\begin{aligned}
	w_t &= \laplace w + m \frac{w^{\chi + 1}}{\int_\Omega w^\chi} && \text{in } \Omega \times (0,T), \\ 
	\grad w \cdot \nu &= 0 && \text{on } \partial \Omega\times(0,T), \\
	w(\cdot, 0) &= w_0 && \text{in }  \Omega
	\end{aligned} 
	\right.
	\end{equation} 
	in $\Ombar\times[0,T)$ with $w_0 \equiv v_0$ and $m \defs \int_\Omega u_0$.
\end{lemma}

For the remainder of this section, we now fix $\chi > 0$, $n \geq 3$ and a bounded convex domain $\Omega \subset \R^n$ with a smooth boundary. We further fix initial data $u_0 \in C^0(\overline{\Omega})$ with $u_0 \geq 0$ and $v_0 \in W^{1,\infty}(\Omega)$ with $v_0 > 0$ and the associated maximally extended solutions $(\ue, \ve)$ constructed in Lemma~\ref{lm:local_exist} on $\Ombar\times[0, \tmaxeps)$ for each $\eps \in (0,1)$. We further assume that there exist $T \in (0,\infty)$, $\eps^\star \in (0,1)$ and $p > \frac{n}{2}$ such that (\ref{eq:assumption}) holds. Lastly, we set $m \defs \int_\Omega u_0$ as a matter of convenience.

As our first step towards a proof of Lemma~\ref{lm:limit_reduction}, we will now derive sufficient bounds for the families $(\ue)_{\eps \in (0,\eps^\star)}$,  $(\ve)_{\eps \in (0,\eps^\star)}$ to allow us to find a suitable sequence $(\eps_j)_{j\in\N} \subset (0, \eps^\star)$, along which both families converge to weak solutions of their natural limit problems by using certain compact embeddings of function spaces.

We start by deriving a uniform lower bound for $\ve$. 
\begin{lemma}\label{lm:ve_lower_bound}
  Suppose there are $T \in (0,\infty)$, $\eps^\star \in (0,1)$ and $p > \frac{n}{2}$ such that \eqref{eq:assumption} holds.
	Then there exists $C > 0$ such that
	\[
	\inf_{x\in\Omega}\ve(x,t) \geq C 
	\]
	for all $t \in (0,T)$ and $\eps \in (0, \eps^\star)$.
\end{lemma}
\begin{proof}
	Using the variation-of-constants representation of $\ve$ and standard maximum-principle estimates for the Neumann heat semigroup $(\ure^{t\laplace})_{t\geq 0}$, we immediately see that
	\[
	\ve(\cdot, t) = \ure^{t(\laplace - 1)}v_0 + \int_0^t \ure^{(t-s)(\laplace - 1)} \ue(\cdot, s)\ve(\cdot, s) \ds \geq \ure^{t(\laplace - 1)}v_0 \geq \ure^{-t} \inf_{x\in\Omega} v_0(x) \geq \ure^{-T} \inf_{x\in\Omega} v_0(x)
	\]
	for all $t \in (0,T)$ and $\eps \in (0, \eps^\star)$. This completes the proof.
\end{proof}
As our next step, we will now establish a fairly strong upper bound for $\ve$ based on the bounds for $\ue$ provided to us by our assumption (\ref{eq:assumption}). As we do not have any mass conversation properties available to us for the second solution components, we will begin by first establishing an $L^2(\Omega)$ baseline bound by testing the second equation in (\ref{problem}) with $\ve$ and then applying Ehrling's lemma. We then use this baseline to establish a much stronger bound of type $W^{1,q}(\Omega)$ by applying well-known smoothing estimates for the Neumann heat semigroup $(\ure^{t\laplace})_{t > 0}$ to the variation-of-constants representation of $\ve$.
\begin{lemma}\label{lm:ve_w1q_bound}
  Suppose there are $T \in (0,\infty)$, $\eps^\star \in (0,1)$ and $p > \frac{n}{2}$ such that \eqref{eq:assumption} holds.
	Then there exist $q > n$ and $C > 0$ such that
	\[
		\|\ve(\cdot,t)\|_{W^{1,q}(\Omega)} \leq C
    \quad \text{and} \quad
		\|\ve(\cdot,t)\|_{L^\infty(\Omega)} \leq C
	\]
	for all $t \in (0,T)$ and $\eps \in (0,\eps^\star)$. 
\end{lemma}
\begin{proof}
	Using assumption (\ref{eq:assumption}), we start by fixing $K_1 > 0$ such that
	\[
	\|\ue(\cdot, t)\|_\L{p} \leq K_1
  \qquad \text{for all $t \in (0,T)$ and $\eps \in (0,\eps^\star)$}.
	\]
	Given this and setting $p' \defs \frac{p}{p-1}$, we can use Ehrling's lemma applied to the triple of spaces $W^{1,2}(\Omega)\hookrightarrow\hookrightarrow L^{2p'}(\Omega) \hookrightarrow L^2(\Omega)$ to fix $K_2 > 0$ such that
	\[
	\|\varphi\|^2_\L{2p'} \leq \frac{1}{K_1} \int_\Omega |\grad \varphi|^2 + K_2 \int_\Omega \varphi^2
  \qquad \text{for all $\varphi \in W^{1,2}(\Omega)$}.
	\]
	We note that this first embedding in the above triple is compact because $p > \frac n2$ entails $2p' = \frac{2p}{p-1} < \frac{2n}{n-2}$.

	To establish a baseline for later arguments, we test the second equation in (\ref{problem}) with $\ve$ to gain
	\begin{align*}
	\frac12 \ddt \int_\Omega \ve^2 &= -\int_\Omega |\grad \ve|^2 - \int_\Omega \ve^2 + \int_\Omega \ve^2 \ue \leq -\int_\Omega |\grad \ve|^2 + \|\ve\|^2_\L{2p'}\|\ue\|_\L{p} \\
	&\leq -\int_\Omega |\grad \ve|^2 + K_1\|\ve\|^2_\L{2p'} \leq K_1 K_2\int_\Omega \ve^2
	\end{align*}
	for all $t\in(0,T)$ and $\eps \in (0,\eps^\star)$. This then lets us fix $K_3 > 0$ such that
	\[
	\|\ve(\cdot, t)\|_\L{2} \leq K_3
	\]
	for all $t\in(0,T)$ and $\eps \in (0,\eps^\star)$ as it indicates that growth of the term is at most exponential.

	Because $p > \frac{n}{2}$, we can fix $r \in (\frac{n}{2}, p)$ such that $\frac{pr}{p-r} > 2$. This in turn implies that $\frac{nr}{(n-r)_+} > n$, allowing us to also fix $q \in (n, \frac{nr}{(n-r)_+})$. We further let $M_\eps(t) \defs \sup_{s\in(0,t)} \|\ve(\cdot,s)\|_{W^{1,q}} < \infty$ for all $t \in (0,T)$ and $\eps \in (0,1)$. Due to standard embedding properties of Sobolev spaces there then exists $K_4 > 0$ such that
	\begin{equation}\label{eq:ve_linf_sobolev_embedding}
		\sup_{s\in(0,t)} \|\ve(\cdot,s)\|_{L^\infty(\Omega)} \leq K_4 M_\eps(t) 
	  \qquad \text{for all $t \in (0,T) $ and $\eps \in (0,\eps^\star)$}.
	\end{equation}
	Using the well-known variation-of-constants representation of $\ve$ in combination with standard smoothing estimates for the Neumann heat semigroup $(\ure^{t\laplace})_{t \geq 0}$ (cf.\ \cite[Lemma~1.3]{WinklerAggregationVsGlobal2010}), we can estimate that
	\begin{align*}
	\|\ve(\cdot, t)\|_{W^{1,q}(\Omega)} &\leq \left\|\ure^{t(\laplace-1)}v_0 \right\|_{W^{1, q}(\Omega)} + \int_0^t \left\| \ure^{(t-s)(\laplace -  1)} \ue(\cdot, s) \ve(\cdot, s) \ds \; \right\|_{W^{1,q}(\Omega)} \\
	&\leq K_5 \|v_0\|_{W^{1,q}(\Omega)} + K_5\int_0^t (1 + (t-s)^{-\frac{1}{2} - \frac{n}{2}(\frac{1}{r} - \frac{1}{q})} )\|(\ue \ve)(\cdot, s) \|_\L{r} \ds
	\end{align*}
	for all $t \in (0,T)$ and $\eps \in (0,\eps^\star)$ with some appropriate $K_5 > 0$.
  Setting $\alpha \defs 1 - \frac{2(p-r)}{pr} \in (0,1)$, we further make use of Hölder's inequality and recall the defining properties of $K_1$, $K_3$ and $K_4$ to obtain
  \begin{align*}
          \sup_{s \in (0, t)} \|(\ue \ve)(\cdot, s) \|_\L{r}
    &\le  \sup_{s \in (0, t)} \|\ue(\cdot, s)\|_\L{p} \| \ve(\cdot, s) \|_\L{\frac{pr}{p-r}} \\
    &\le  \sup_{s \in (0, t)} \|\ue(\cdot, s)\|_\L{p} \|\ve(\cdot, s)\|_\L{\infty}^\alpha \| \ve(\cdot, s) \|^{1-\alpha}_\L{2}
    \le   K_1 K_3^{1-\alpha} K_4^\alpha M_\eps^\alpha(t)
  \end{align*}
	for all $t \in (0,T)$ and $\eps \in (0,\eps^\star)$.
  Moreover,
  \begin{align*}
        \int_0^t (1 + (t-s)^{-\frac{1}{2} - \frac{n}{2}(\frac{1}{r} - \frac{1}{q})})\ds
    \le \int_0^T (1 + s^{-\frac{1}{2} - \frac{n}{2}(\frac{1}{r} - \frac{1}{q})})\ds
    \lt \infty
  \end{align*}
	for all $t \in (0,T)$ since by our choice of $q$ we ensured that $\frac{n}{2}(\frac{1}{r} - \frac{1}{q}) < \frac{1}{2}$.
  Combining these estimates, we conclude that there is $K_6 \gt 0$ such that
	\[
		M_\eps(t) \leq K_6 + K_6 M_\eps^\alpha(t)
    \qquad \text{for all $t \in (0,T)$ and $\eps \in (0,\eps^\star)$},
	\]
	which together with (\ref{eq:ve_linf_sobolev_embedding}) is sufficient to complete the proof as $\alpha \in (0,1)$.
\end{proof}
Next, we want to establish some space-time integral bounds for the first solution components $\ue$. To do this, we test the first equation in (\ref{problem}) with $\eps\ue$ and then combine the resulting differential inequality with the already established bounds for $\ve$.
\begin{lemma}\label{lm:grad_ue2_sp_bound}
  Suppose there are $T \in (0,\infty)$, $\eps^\star \in (0,1)$ and $p > \frac{n}{2}$ such that \eqref{eq:assumption} holds.
	Then there exists $C > 0$ such that
	\[
	\int_0^T \int_\Omega |\grad \ue|^2 \leq C
	\]
	for all $\eps \in (0,\eps^\star)$.
\end{lemma}
\begin{proof}
	Using Lemma~\ref{lm:ve_lower_bound} and Lemma~\ref{lm:ve_w1q_bound}, we can fix $q > n > 2$ and $K_1 > 0$  such that
	\[
	\left\|\frac{1}{\ve(\cdot, t)}\right\|_\L{\infty} \leq K_1 \quad \text{ and } \quad	\left\|\ve(\cdot, t)\right\|_{W^{1,q}(\Omega)} \leq K_1
	\]
	for all $t \in (0,T)$ and $\eps \in (0,\eps^\star)$. Since $1 < \frac{2q}{q-2} < \frac{2n}{n-2}$ and thus $\sob12 \embed \embed \leb{\frac{2q}{q-2}} \embed \leb1$,
  Ehrling's lemma combined with \eqref{eq:mass_bound} allows us to fix $K_2 > 0$ such that
	\[
	  K_1^4 \chi^2 \| \ue(\cdot, t) \|_\L{\frac{2q}{q-2}}^2 \leq \frac12 \|\grad \ue(\cdot, t)\|^2_\L{2} + K_2 
	  \qquad \text{for all $t \in (0,T)$ and $\eps \in (0,\eps^\star)$}.
	\]
	Testing the first equation in (\ref{problem}) with $\eps \ue$ and applying partial integration as well as Young's inequality yields
	\begin{align*}
	\frac{\eps}{2} \ddt \int_\Omega \ue^2 + \int_\Omega |\grad \ue|^2 &= \chi\int_\Omega  \frac{\ue}{\ve} \grad \ue \cdot \grad \ve  - \eps^2 \int_\Omega \ue^2 \ve \leq  \frac{1}{2} \int_\Omega |\grad \ue|^2 + \frac{\chi^2}{2} \int_\Omega \frac{\ue^2}{\ve^2} |\grad \ve|^2
  \qquad \text{in $(0, T)$},
	\end{align*}
	which when combined with Hölder's inequality implies 
	\begin{equation*}
          \eps \ddt \int_\Omega \ue^2 + \int_\Omega |\grad \ue|^2
    \le   \chi^2 \int_\Omega \frac{\ue^2}{\ve^2} |\grad \ve|^2
    \le   K_1^2 \chi^2 \left( \intom \ue^\frac{2q}{q-2} \right)^\frac{q-2}{q} \left( \intom |\grad \ve|^q \right)^\frac2q
    \le   \frac{1}{2} \int_\Omega |\grad \ue|^2 + K_2
	\end{equation*}
	for all $t \in (0,T)$ and $\eps \in (0,\eps^\star)$.
	Integrating with regards to the time variable then completes the proof.
\end{proof}
To eventually make the Aubin--Lions lemma applicable, we will derive one last set of a priori space-time bounds for $\ve$. To this end, we now test the second equation in (\ref{problem}) with $-\laplace \ve$ and $v_{\eps t}$.
\begin{lemma}\label{lm:ve_higher_sp_bounds}
  Suppose there are $T \in (0,\infty)$, $\eps^\star \in (0,1)$ and $p > \frac{n}{2}$ such that \eqref{eq:assumption} holds.
	Then there exists $C > 0$ such that
	\[
	\int_0^T\int_\Omega |\laplace \ve|^2 \leq C
  \quad \text{and} \quad
	\int_0^T\int_\Omega v_{\eps t}^2 \leq C
	\]
	for all $\eps \in (0,\eps^\star)$.
\end{lemma}
\begin{proof}
	Given Lemma~\ref{lm:ve_w1q_bound}, we can fix $K_1 > 0$ such that
	\[
	\|\ve(\cdot, t)\|_\L{\infty} \leq K_1
  \qquad \text{for all $t \in (0,T)$ and $\eps \in (0,\eps^\star)$}.
	\]
	Since $\sob12 \embed\embed \leb2 \embed \leb 1$, Ehrling's lemma in combination with the mass boundedness property (\ref{eq:mass_bound}) allows us to further fix $K_2 > 0$ such that
	\[
	\int_\Omega \ue^2 \leq \int_\Omega |\grad \ue|^2 + K_2
  \qquad \text{for all $t \in (0,T)$ and $\eps \in (0,\eps^\star)$}.
	\] 
	Testing the second equation in (\ref{problem}) with $-\laplace \ve$ and applying partial integration as well as Young's inequality immediately gives us
	\begin{align*}
	\frac{1}{2} \ddt \int_\Omega |\grad \ve|^2 + \int_\Omega |\laplace \ve|^2 + \int_\Omega |\grad \ve|^2 \leq -\int_\Omega \ue\ve \laplace \ve \leq \frac{1}{2}\int_\Omega |\laplace \ve|^2 + \frac{1}{2}\int_\Omega \ve^2 \ue^2.
	\end{align*}
	for all $t \in (0,T)$ and $\eps \in (0,1)$. Similarly,
	\begin{align*}
	\int_\Omega v_{\eps t}^2 = \int_\Omega (\laplace \ve - \ve + \ve\ue)^2  \leq 2\int_\Omega |\laplace \ve|^2 + 4\int_\Omega \ve^2 + 4\int_\Omega \ve^2\ue^2
	\end{align*}
	for all $t \in (0,T)$ and $\eps \in (0,1)$. Combining both of the above differential inequalities in an appropriate fashion and applying the prior established bounds and functional inequalities then yields
	\begin{align*}
	\ddt \int_\Omega |\grad \ve|^2 + \frac{1}{2}\int_\Omega |\laplace \ve|^2 + \frac{1}{4} \int_\Omega v_{\eps t}^2 &\leq \int_\Omega \ve^2 + 2\int_\Omega \ve^2 \ue^2  \\
	&\leq K_1^2|\Omega| + 2K_1^2\int_\Omega \ue^2\\
	&\leq 2K_1^2 \int_\Omega |\grad \ue|^2 + K_3 
	\end{align*}
	for all $t \in (0,T)$ and $\eps \in (0,\eps^\star)$ with $K_3 \defs K_1^2|\Omega| + 2K_1^2K_2$. Given the bound derived in Lemma~\ref{lm:grad_ue2_sp_bound}, time integration again yields our desired result.
\end{proof}
We have now derived all the bounds necessary to construct our desired functions $u$ and $v$ as limits of the functions $\ue$ and $\ve$, respectively, with convergence properties sufficiently strong to translate some (weak) solution properties from $\ue$ and $\ve$ to said limit functions. Note that the limit function $u$ will of course no longer solve a parabolic but an elliptic problem, which can be obtained by multiplying the first equation in (\ref{problem}) by $\eps$ and then setting $\eps$ to zero.
\begin{lemma}\label{lm:limit_functions}
  Suppose there are $T \in (0,\infty)$, $\eps^\star \in (0,1)$ and $p > \frac{n}{2}$ such that \eqref{eq:assumption} holds.
	Then there exist functions
	\begin{align*}
	u &\in L^2((0,T);W^{1,2}(\Omega)), \\
	v &\in L^2((0,T);W^{1,2}(\Omega)) \cap L^\infty(\Omega\times(0,T))
	\end{align*}
	with $\frac{1}{v} \in  L^\infty(\Omega\times(0,T))$
	and a null sequence $(\eps_j)_{j\in\N} \subset (0,\eps^\star)$ such that
	\begin{alignat}{2}
	\ue       & \rightharpoonup u       &&\qquad \text{ in } L^2(\Omega\times(0,T)) \label{eq:u_l2_conv}, \\
	\grad \ue & \rightharpoonup \grad u &&\qquad \text{ in } L^2(\Omega\times(0,T))\label{eq:grad_u_l2_conv}, \\
	\ve       & \rightarrow v           &&\qquad \text{ in } L^2(\Omega\times(0,T)) \text{ and a.e.\ in }  \Omega\times(0,T) \label{eq:v_l2_conv} \text{ and } \\
	\grad \ve & \rightarrow \grad v     &&\qquad \text{ in } L^2(\Omega\times(0,T)) \label{eq:grad_v_l2_conv}
	\end{alignat}
	as $\eps = \eps_j \searrow 0$. Furthermore, $u$ is a weak solution of the elliptic counterpart to the first equation in (\ref{problem}) in the sense that
	\begin{equation}\label{eq:u_weak_solution}
	\int_0^T\int_\Omega \grad u \cdot \grad \varphi = \chi \int_0^T\int_\Omega \frac{u}{v} \grad v \cdot \grad \varphi
	\end{equation}
	and $v$ is a weak solution to the second equation in (\ref{problem}) in the sense that
	\begin{equation}\label{eq:v_weak_solution}
	-\int_0^T\int_\Omega v \varphi_t - \int_\Omega v_0 \varphi(\cdot, 0) = -\int_0^T\int_\Omega \grad v \cdot \grad \varphi - \int_0^T \int_\Omega v\varphi + \int_0^T \int_\Omega v u \varphi
	\end{equation}
	for all $\varphi \in C^\infty(\overline{\Omega}\times[0,T))$ with compact support.
\end{lemma}
\begin{proof}
	Due to the mass boundedness property (\ref{eq:mass_bound}) and Lemma~\ref{lm:grad_ue2_sp_bound} we know that the family $(\ue)_{\eps \in (0,\eps^\star)}$ is uniformly bounded in $L^2((0,T);W^{1,2}(\Omega))$. Noting that bounded sets in Sobolev spaces are already compact regarding the weak topology in the very same spaces, this already allows us to find a function $u \in L^2((0,T);W^{1,2}(\Omega))$ and a null sequence $(\eps_j)_{j\in\N} \subset (0, \eps^\star)$ such that (\ref{eq:u_l2_conv}) and (\ref{eq:grad_u_l2_conv}) hold as $\eps = \eps_j \searrow 0$ by a standard subsequence extraction argument.

	Further because of Lemma~\ref{lm:ve_w1q_bound} and Lemma~\ref{lm:ve_higher_sp_bounds} potentially combined with a standard elliptic regularity estimate (cf.\ \cite[Lemma 19.1]{FriedmanPartialDifferentialEquations1969}), we know that the family $(\ve)_{\eps \in (0,\eps^\star)}$ is uniformly bounded in $L^2((0,T);W^{2,2}(\Omega))$ and the family $(v_{\eps t})_{\eps \in (0,\eps^\star)}$ is uniformly bounded in $L^2((0,T);L^2(\Omega))$. Therefore, we can use the Aubin--Lions lemma (cf.\ \cite{TemamNavierStokesEquationsTheory1977}) applied to the triple of spaces $W^{2,2}(\Omega)\hookrightarrow\hookrightarrow W^{1,2}(\Omega)\hookrightarrow L^2(\Omega)$ to construct a function $v \in L^2((0,T);W^{1,2}(\Omega)$ such that (\ref{eq:v_l2_conv}) and (\ref{eq:grad_v_l2_conv}) hold as $\eps = \eps_j \searrow 0$ by replacing the previously established sequence $(\eps_j)_{j\in\N}$ with a subsequence extracted from it. Note that the almost everywhere pointwise convergence mentioned in (\ref{eq:v_l2_conv}) is not a direct consequence of the Aubin--Lions lemma, but can be established by potentially a third subsequence extraction as it is well-known that $L^2(\Omega\times(0,T))$ convergence implies a.e.\ pointwise convergence along a subsequence. It is then this pointwise convergence property combined with the uniform upper and lower bounds from Lemma~\ref{lm:ve_lower_bound} and Lemma~\ref{lm:ve_w1q_bound} that ensures that $v$ and $\frac{1}{v}$ are elements of $L^\infty(\Omega\times(0,T))$.

	Having now proven all necessary regularity and convergence properties for $u$ and $v$, we only need to verify that our desired weak solution properties hold as well. Testing with $\eps\varphi$ and integrating by parts shows that all $\ue$ are weak solutions to a variant of the first equation in (\ref{problem}) in the sense that 
	\begin{equation}\label{eq:ue_weak_solution}
	-\eps \int_0^T\int_\Omega \ue \varphi_t - \eps\int_\Omega u_0 \varphi(\cdot, 0) + \int_0^T\int_\Omega \grad \ue \cdot \grad \varphi = \chi \int_0^T\int_\Omega \frac{\ue}{\ve} \grad \ve \cdot \grad \varphi - \eps^2 \int_0^T\int_\Omega \ue \ve \varphi
	\end{equation}
	for all $\eps \in (0,1)$ and $\varphi \in C^\infty(\Ombar\times[0,T))$ with compact support. We now note that for fixed $\varphi$, the mass boundedness property (\ref{eq:mass_bound}) as well as the $L^\infty(\Omega)$ bound established in Lemma~\ref{lm:ve_w1q_bound} are already sufficient to ensure that the families $(|\int_0^T\int_\Omega \ue \varphi_t|)_{\eps \in (0, 1)}$, $(|\int_\Omega u_0 \varphi(\cdot, 0)|)_{\eps \in (0, 1)}$ and $(|\int_0^T\int_\Omega \ue \ve \varphi|)_{\eps \in (0, 1)}$ are uniformly bounded. Therefore, $\eps \int_0^T\int_\Omega \ue \varphi_t$, $\eps\int_\Omega u_0 \varphi(\cdot, 0)$ and $\eps^2 \int_0^T\int_\Omega \ue \ve \varphi$ converge to zero as $\eps \searrow 0$ for fixed $\varphi$. It is further easy to check, that the remaining terms converge to their counterparts without $\eps$ given the convergence properties in (\ref{eq:u_l2_conv})--(\ref{eq:grad_v_l2_conv}). As such it is sufficient to take the limit $\eps = \eps_j \searrow 0$ in (\ref{eq:ue_weak_solution}) to gain the weak solution property (\ref{eq:u_weak_solution}) for $u$. Similarly, each $\ve$ already fulfills the weak solution property (\ref{eq:v_weak_solution}) and it is again easy to verify that the convergence properties (\ref{eq:u_l2_conv})--(\ref{eq:grad_v_l2_conv}) are sufficient to retain said solution property as $\eps = \eps_j \searrow 0$.  
\end{proof}

It is a common approach when analyzing a parabolic equation coupled to an elliptic one, to essentially reduce the problem to a single parabolic equation if at all possible. As such, our next step toward our proof of Lemma~\ref{lm:limit_reduction} will be the derivation of an almost everywhere explicit formula for $u$ in terms of $v$, which can then be applied to the weak formulation of the second subproblem to achieve such a reduction.

To start, we adapt the argument used in \cite[Lemma 3.1]{WinklerSingularLimit2021} to show that the limit solution $u$ has constant mass $m$.
\begin{lemma}\label{lm:u_const_mass}
  Suppose there are $T \in (0,\infty)$, $\eps^\star \in (0,1)$ and $p > \frac{n}{2}$ such that \eqref{eq:assumption} holds
  and let $u$ be the function constructed in Lemma~\ref{lm:limit_functions}. Then there exists a null set $N\subset (0,T)$ such that 
  \[
    \int_\Omega u(\cdot,t) = m
  \]
  for all $t\in(0,T)\setminus N$.
\end{lemma}
\begin{proof}
  Using the convergence properties \eqref{eq:u_l2_conv} and \eqref{eq:v_l2_conv} from Lemma~\ref{lm:limit_functions}, we first note that (with $v$ as in Lemma~\ref{lm:limit_functions})
  \[
    \int_t^{t+\delta}\int_\Omega \ue \ve 
    \rightarrow  
    \int_t^{t+\delta}\int_\Omega u v
  \]
  and thus by \eqref{eq:mass_bound} that
  \begin{equation*}
    \int_t^{t+\delta}\int_\Omega \ue = \delta m - \eps \int_t^{t+\delta}\int_\Omega \ue \ve \rightarrow \delta m 
  \end{equation*}
  for all $\delta \in (0,1)$ and $t \in [0,T-\delta]$ as $\eps = \eps_j \searrow 0$. But as \eqref{eq:u_l2_conv} also implies that 
  \[
    \int_t^{t+\delta}\int_\Omega \ue 
    \rightarrow  
    \int_t^{t+\delta}\int_\Omega u
  \]
  for all $\delta \in (0,1)$ and $t \in [0,T-\delta]$ as $\eps = \eps_j \searrow 0$,
  we can conclude that 
  \begin{equation}\label{eq:u_space_time_mass_conservation}
    \int_t^{t+\delta}\int_\Omega u = \delta m
  \end{equation}
  for all $\delta \in (0,1)$ and $t \in [0,T-\delta]$ by the uniqueness property of limits. 
  \\[0.5em]
  Let now $N \subset (0,T)$ be the set of times $t \in (0,T)$ that are not Lebesgue points of the map $t \mapsto \int_\Omega u(\cdot, t)$. Using \eqref{eq:u_space_time_mass_conservation} combined with the Lebesgue point property, we then see that  
  \[
    \int_\Omega u(\cdot, t) = \lim_{\delta \searrow 0} \frac{1}{\delta}\int_t^{t+\delta}\int_\Omega u = m
  \]
  for all $t \in (0,T)\setminus N$. As further $u\in L^1((0,T); L^1(\Omega))$ implies that $N$ is a null set, this completes the proof.
\end{proof}
Having established all the necessary preliminaries, we now derive an explicit formula for $u$ in terms of $v$.

\begin{lemma}\label{lm:u_explicit_form}
  Suppose there are $T \in (0,\infty)$, $\eps^\star \in (0,1)$ and $p > \frac{n}{2}$ such that \eqref{eq:assumption} holds
  and let $u$, $v$ be the functions constructed in Lemma~\ref{lm:limit_functions}. Then there exists a null set $N \subset (0,T)$ such that 
	\begin{equation}\label{eq:u_explicit_form}
	u(\cdot,t) = m \frac{v^\chi(\cdot,t)}{\int_\Omega v^\chi(\cdot,t)} \qquad \text{a.e.\ in $\Omega$}
	\end{equation}
	for all $t \in (0,T)\setminus N$.
\end{lemma}
\begin{proof}
The derivation of such an explicit formula for $u$ has already been discussed in detail in a closely related setting in Section~3 of \cite{WinklerSingularLimit2021}. Our setting only slightly differs from the one discussed in the reference in regards to the equation solved by $v$, which does not come into play in the relevant arguments to derive the explicit representation for $u$. Given this, we will only briefly sketch out the relevant ideas for said derivation and refer to  \cite{WinklerSingularLimit2021} for a more in-depth discussion as it seems unnecessary to basically reiterate the exact arguments from the reference.

The argument mainly rests on two key ideas:
\begin{itemize}
	\item The mass conservation property proven in Lemma~\ref{lm:u_const_mass} in combination with the already established regularity properties for $u$ and $v$ in Lemma~\ref{lm:limit_functions} is sufficient to conclude from the weak solution property (\ref{eq:u_weak_solution}) that $u(x,t)$ is positive for almost every $x\in\Omega$ and almost every $t \in \Omega$ (cf.\ \cite[Lemma 3.6]{WinklerSingularLimit2021}).
	\item Further, if we could set $\varphi = \ln(u) - \chi \ln(v)$ as a test function in (\ref{eq:u_weak_solution}), it would result in
	\[
	\int_0^T\int_\Omega u|\grad \ln(u) - \chi \grad \ln(v)|^2 = 0,
	\]
	which in turn would give us that $\ln(u) = \chi\ln(v) + C$ or further that $u = C v^\chi$ almost everywhere due to $u$ being sufficiently positive. As $\int_\Omega u = m$ for almost all times, we would gain that $C = m\left(\int_\Omega v^\chi\right)^{-1}$, which would give us our desired representation (\ref{eq:u_explicit_form}) almost everywhere. It is, of course, necessary to make this idea precise by means of appropriate approximation arguments given that the $\varphi$ proposed above is not a sufficiently regular test function (cf.\ \cite[Lemma 3.7]{WinklerSingularLimit2021}). Notably it is only this approximation argument that necessitates for $\Omega$ to be convex.
\qedhere
\end{itemize} 
\end{proof}

This representation for $u$ for almost every time $t$ now allows for a reduction of the two weak solution properties in (\ref{eq:u_weak_solution}) and (\ref{eq:v_weak_solution}) to a single weak formulation for the second solution component. Given the global upper and lower bounds for $v$ already established, this in turn allows us to show that $v$ is in fact already a classical solution by use of some standard parabolic regularity theory. By then setting $w \defs \ure^t v$, we gain the classical solution to (\ref{limit_problem}) desired in Lemma~\ref{lm:limit_reduction}.

\begin{proof}[Proof of Lemma~\ref{lm:limit_reduction}]
	According to Lemma~\ref{lm:limit_functions} and Lemma~\ref{lm:u_explicit_form}, the function $v \in L^2((0,T); W^{1,2}(\Omega))$ constructed in Lemma~\ref{lm:limit_functions} is in fact a standard weak solution of the parabolic Neumann problem
	\begin{equation}\label{eq:v_classical_solution}
	\left\{ \;
	\begin{aligned}
	v_t &= \laplace v - v + f(x,t) && \text{in } \Omega\times(0,T),\\
	\grad v \cdot \nu &= 0 && \text{on } \partial\Omega\times(0,T),\\  
	v(\cdot, 0) &= v_0 && \text{in } \Omega\\  
	\end{aligned}
	\right.
	\end{equation}
	with 
	\[
	f(x,t) \defs m\,\frac{v^{\chi + 1}(x,t)}{\int_\Omega v^\chi(\cdot,t)} \quad \quad \text{ for all } (x,t) \in \Omega\times [0,T).
	\]
	This makes the function $v$ accessible to standard parabolic regularity theory. Because both $v$ and $\frac{1}{v}$ belong to $L^\infty(\Omega\times(0,T))$ by Lemma~\ref{lm:limit_functions}, we know that $f\in L^\infty(\Omega\times(0,T))$. As such, the Hölder regularity results from \cite{PorzioVespriHoelder} are applicable and immediately give us that $v \in C^{\alpha, \frac{\alpha}{2}}(\overline{\Omega}\times[0,T])$. Together with the already mentioned bounds for $v$, this directly implies $f\in C^{\alpha, \frac{\alpha}{2}}(\overline{\Omega}\times[0,T])$. Another application of parabolic regularity theory then yields that $v$ is an element of $C^{2,1}(\overline{\Omega}\times(0,T))$ and therefore a classical solution to (\ref{eq:v_classical_solution}) (cf.\ \cite{LadyzenskajaLinearQuasilinearEquations1988}).

	If we now set $w \defs \ure^t v$ on $\overline{\Omega}\times[0,T]$, we can compute that $w$ fulfills the partial differential equation
	\[
	w_t = \ure^t v + \ure^t v_t = \ure^t\laplace v + m \ure^t \frac{v^{\chi+1}}{\int_\Omega v^\chi} = \laplace w + m \frac{w^{\chi+1}}{\int_\Omega w^\chi} \qquad \text{ in } \Omega\times(0,T)
	\]
	with Neumann zero as well as the initial data condition $w(\cdot, 0) = v(\cdot, 0) = v_0$. As the function $w$ is further of the necessary regularity, this completes the proof.
\end{proof}

\section{Blow-up in the limit problem}\label{sec:limit_blowup}
Throughout this section, we fix $R > 0$, $n \ge 3$, $\Omega \defs B_R(0)$ and $\chi > 0$.
Moreover, whenever convenient, we will switch to radial notation and thus for instance write $z(|x|)$ instead of $z(x)$ and denote the derivative with respect to $r = |x|$ by $z_r$ for radially symmetric functions $z$.

We have seen in the previous section that, under the uniform boundedness condition (\ref{eq:assumption}) for the families of solutions to the problem (\ref{problem}) on some time interval $(0,T)$, it is in fact possible to construct a classical solution to what is essentially a scalar variant of its limit problem on the same time interval. As such, this section will now be devoted to analyzing said problem,
	\begin{equation}\label{prob:w}
  \left\{\; \begin{aligned}
	w_t &= \Delta w + m \left( \intom w^\chi \right)^{-1} w^{\chi+1} && \text{in $\Omega \times (0, T)$}, \\
	\nabla w \cdot \nu &= 0                                          && \text{on $\partial \Omega \times (0, T)$}, \\
	w(\cdot, 0) &= w_0                                               && \text{in $\Omega$}
  \end{aligned} \right.
	\end{equation}
for given $m > 0$.
More specifically for any given $T > 0$, our goal is to construct (radial) nonnegative initial data $w_0$ such that the associated solution blows up before said time $T$ provided $m$ is sufficiently large.

To place the arguments of this section on a reasonably solid foundation, we begin by giving an existence and crucially uniqueness result for classical solutions of (\ref{prob:w}) accompanied by an appropriate blow-up criterion. The uniqueness property is of particular interest as it allows us to identify the solution constructed in the previous section with the blow-up solution constructed in this section. We further give some straightforward additional regularity properties of these solutions that are necessary for some of the methods employed later in this section.
\begin{lemma}\label{lm:ex_w}
  Let $m \gt 0$ and
  suppose that $1 \le w_0 \in \con2$ with $\nabla w_0 \cdot \nu = 0$ on $\partial \Omega$ is radially symmetric.
  Then there exist $\tmax \equiv \tmax(m, w_0) \in (0, \infty]$ and a unique classical solution
  \begin{align*}
    w \in C^0(\Ombar \times [0,\tmax))\cap C^{2, 1}(\Ombar \times (0, \tmax))
  \end{align*}
  of \eqref{prob:w}, which is radially symmetric and fulfills $w \ge 1$ in $\Ombar \times [0, \tmax)$,
  with the property that if $\tmax \lt \infty$, then
  \begin{align*}
    \limsup_{t \nea \tmax} \|w(\cdot, t)\|_{\leb\infty} = \infty.
  \end{align*}
  Moreover, $w_r$ belongs to $C^{2, 1}(\Ombar \times (0, \tmax)) \cap C^0(\Ombar \times [0, \tmax))$.
\end{lemma}
\begin{proof}
  We fix a positive and radially symmetric $w_0 \in C^2({\Ombar})$ with $\partial_\nu w_0 = 0$ on $\partial \Omega$ as well as some $m > 0$.

  Local existence and regularity properties can be shown as in \cite[Example~51.13]{QuittnerSoupletSuperlinearParabolicProblems2019} by essentially using results from \cite{LunardiAnalyticSemigroupsOptimal1995} in appropriate fashion. Given that we need a slightly stronger version of these properties, we will nonetheless give a quick sketch of the necessary arguments:

  To frame our parabolic problem in way accessible to the results in \cite{LunardiAnalyticSemigroupsOptimal1995}, let first $A$ be the sectorial realization of the Neumann Laplacian on $X \defs C^0(\Ombar)$ given in \cite[Corollary~3.1.24 (ii)]{LunardiAnalyticSemigroupsOptimal1995} and \[
  	\mathcal{O}\defs \left\{ \varphi \in X \,\middle|\, \inf_{x\in\Ombar} \varphi > \frac{1}{2} \right\}.
  \] 
  Then the (actually $t$-independent) function
  \[
  	F: [0,\infty)\times\mathcal{O} \rightarrow X, \;\;\;\; (t,\varphi) \mapsto m \left(\int_\Omega \varphi^\chi\right)^{-1} \varphi^{\chi + 1},
  \]
  is continuous and satisfies the local Lipschitz property (7.1.12) from \cite{LunardiAnalyticSemigroupsOptimal1995}. This then allows us to use Theorem~7.1.3~(i) and Definition~7.1.7 from \cite{LunardiAnalyticSemigroupsOptimal1995} to conclude that the problem
  \begin{equation*}
  \left\{\;
  \begin{aligned}
    w_t &= A w + F(t,w) && \text{ in } (0, \infty), \\
    w(0) &= w_0
  \end{aligned}
  \right.
  \end{equation*}
  has a unique mild solution $w \in C^0([0,\tmax);C^0(\Ombar))$ with maximal existence time $\tmax \in (0,\infty]$. By application of \cite[Proposition~7.1.10]{LunardiAnalyticSemigroupsOptimal1995}, we then gain that the solution was in fact classical as well. Further applications of the same proposition for the realizations of $A$ on either $C^1(\Ombar)$ or $C^{1+\alpha}(\Ombar)$ discussed in \cite[Theorem~3.1.26]{LunardiAnalyticSemigroupsOptimal1995} and \cite[Corollary 3.1.32]{LunardiAnalyticSemigroupsOptimal1995}, respectively, then result in our remaining desired regularity properties for $w$ and $w_r$.
  
  As $w$ is a supersolution of the corresponding Neumann heat equation, we further note that $w \geq 1$ in $\Ombar \times [0,\tmax)$ by comparison, which gives us our desired blow-up criterion due to \cite[Proposition 7.1.8]{LunardiAnalyticSemigroupsOptimal1995} and the comment immediately following it. Lastly, radial symmetry of $w$ follows by the already established uniqueness property as any rotated version of $w$ again solves the very same Neumann problem.
\end{proof}
Given this existence result, we can now in fact make the central result of this section described in the introductory paragraph precise.
\begin{lemma}\label{lm:w_arbitrarily_earily_blowup}
There exists $m_0 > 0$ such that for $m > m_0$ the following holds true: For each $T > 0$, there exists radially symmetric initial data $1 \le w_0 \in C^2(\overline{\Omega})$ with $\nabla w_0 \cdot \nu = 0$ on $\partial \Omega$ such that the unique classical solution to (\ref{prob:w}) given by Lemma~\ref{lm:ex_w} with initial data $w_0$ has a maximal existence time $\tmax$ of less than $T$.
\end{lemma}
As our first step in proving the above, we now construct candidates for the initial data $w_0$. The key idea here is that the closer we already start to something similar to a blow-up state, while retaining the ability to derive a uniform upper bound for the dampening term $\int_\Omega w^\chi$ in (\ref{prob:w}), the faster we would expect the solution to blow up, if it does so at all. 

As such, for our initial data we consider a singularity of the type $S(x) = |x|^{-\frac{2}{\chi}}$, which we cut off close to zero to ensure sufficient differentiability and cut off near the boundary to ensure Neumann boundary conditions. We can do this while uniformly bounding the initial data in $L^\chi(\Omega)$ simply because $S \in L^\chi(\Omega)$ already. Further note that we are able to retain some of the convenient properties of $S$ such as $S_r(r) = -\tfrac{2}{\chi}r S^{\chi + 1}(r) < 0$ and $\laplace S(x) = -\frac{2}{\chi} (n - \frac{2}{\chi} - 2) S(x)^{\chi + 1}$ in a uniform, albeit somewhat weakened,  fashion.
\begin{lemma}\label{lm:w0}
  There exist $A, \lambda, \mu \gt 0$ and $W \in C^\infty([\frac R4, R])$ with $W_r \lt 0$ in $(\frac R4, R)$
  such that for all $M \gt 0$,
  there is $1 \le  w_0 \in \con2$ with
  \begin{gather}
    \label{eq:w0:multi}
    w_0(0) \ge M, \quad
    w_{0 r}(R) = 0
    \quad \text{and} \quad
    w_{0 r} \le 0 \text{ in $(0, R)$}, \\
    \label{eq:w0:w0_W}
    w_0(r) = W(r) \qquad \text{for all $r \in [\tfrac R4, R]$}, \\
    \label{wq:w0:lchi}
    \intom w_0^\chi \le A^{-1}, \\
    \label{eq:w0:near_bu}
    \Delta w_0 + \lambda w_0^{\chi+1} \ge 0 \qquad \text{in $\Omega$} \quad \text{and} \\
    \label{eq:w0:comp}
    w_{0 r} + \mu r w_0^{\chi+1} \le 0 \qquad \text{in $(0, \tfrac{R}{2})$}.
  \end{gather}
\end{lemma}
\begin{proof}
  This can be proven similarly to \cite[Lemma~44.10]{QuittnerSoupletSuperlinearParabolicProblems2019}.
  However, as some small modifications are needed, we still choose to sketch the proof with a particular focus on the pertinent changes.

  We let $\alpha \defs \frac{2}{\chi} \gt 0$ and fix a positive function $W \in C^\infty((0, R])$ with
  \begin{align*}
    W(r) = R^\alpha r^{-\alpha} \text{ for $r \in (0, \tfrac{R}{2})$}, \quad
    W(R) = 1, \quad
    W_r \lt 0 \text{ in $(0, R)$} \quad \text{and} \quad
    W_r(R) = 0.
  \end{align*}
  Moreover letting $\delta \defs \min\{M^{-\frac1\alpha}, \frac{R}{4}\}$, a direct computation shows that
  \begin{align*}
    w_0 \colon \Ombar \ra [0, \infty), \quad
    r \mapsto
    \begin{cases}
      R^\alpha \delta^{-\alpha}(
        1
        + \frac{\alpha(\alpha+5)}{6}
        - \frac{\alpha(\alpha+3)}{2} (\frac{r}{\delta})^2
        + \frac{\alpha(\alpha+2)}{3} (\frac{r}{\delta})^3
      ), & 0 \le r \le \delta, \\
      W(r), & \delta \lt r \le R
    \end{cases}
  \end{align*}
  belongs to $C^2([0,R])$ with $1 \le w_0 \le W$ in $(0, R)$.
  Thus, \eqref{eq:w0:multi} and \eqref{eq:w0:w0_W} are fulfilled.
  Moreover, since $\alpha \chi = 2 < n$ implies that $\int_{B_{R/2}(0)} W^\chi(|x|) \dx = R^{\alpha \chi}\int_{B_{R/2}(0)} |x|^{-\alpha \chi}\dx$ is finite,
  \eqref{wq:w0:lchi} holds for $A \defs (\intom W^\chi)^{-1} \in (0, \infty)$.

  As it is fairly easy to see that $W$ satisfies (\ref{eq:w0:near_bu}) and (\ref{eq:w0:comp}) in $(0, R)$  by a straightforward computation, we will focus our verification efforts for both of these properties for $w_0$ on the interval $(0,\delta)$ to make sure that the involved constants are in fact independent of $\delta$ and hence also of $M$. We first note that
  \[
  	 w_{0 r}(r) = R^\alpha \delta^{-\alpha-1} \alpha \left( - (\alpha + 3) \left(\tfrac{r}{\delta}\right) + (\alpha + 2) \left(\tfrac{r}{\delta}\right)^2\right) 
  	 %
  	 %
     \begin{cases}
  	  \leq -R^\alpha \delta^{-\alpha-2} r \alpha, \\
      \geq -R^\alpha \delta^{-\alpha-2} r \alpha(\alpha + 3)
     \end{cases}
  \] 
  for all $r \in (0, \delta)$
  and similarly
  \[
  	\laplace w_0 = {w_0}_{rr} + \frac{n-1}{r} {w_0}_r \geq -R^\alpha\delta^{-\alpha - 2} \alpha (\alpha + 3) n
  \]
  in $(0,\delta)$. As further
  \[
  	R^{\alpha - 2} \delta^{-\alpha - 2} = w_0^{\chi + 1}(\delta) \leq  w_0^{\chi + 1}(r) \leq w_0^{\chi + 1}(0) = R^{\alpha - 2} \delta^{-\alpha - 2} \left(1 + \tfrac{\alpha(\alpha + 5)}{6}\right)^{\chi + 1},
  \]
  for all $r \in (0,\delta)$, the above estimates immediately yield
  \begin{align*}
    \Delta w_0 + \lambda w_0^{\chi+1} \ge 0 \text{ in $B_\delta(0)$}
    \quad \text{and} \quad
    w_{0 r} + \mu r w_0^{\chi+1} \le 0 \text{ in $(0, \delta)$}
  \end{align*} 
  for $\lambda \defs \alpha(\alpha + 3)n R^2$ and $\mu \defs \alpha ( 1 + \frac{\alpha(\alpha + 5)}{6})^{-\chi - 1}R^2$.
  Upon enlarging $\lambda$ and $-\mu$ to account for $W$, if necessary, this entails \eqref{eq:w0:near_bu} and \eqref{eq:w0:comp}.
\end{proof}

To streamline later arguments, we henceforth fix $A, \lambda, \mu > 0$ and $W$ as given by Lemma~\ref{lm:w0}.
As our next step in proving Lemma~\ref{lm:w_arbitrarily_earily_blowup}, we now show that the solution $w(\cdot, \cdot; w_0, m)$ of \eqref{prob:w} emanating from the initial data $w_0$ given by Lemma~\ref{lm:w0} blows up arbitrarily early if $m$ and $M$ are sufficiently large.
The main ingredient of our proof is obtaining a lower bound for the nonlocal term
\begin{align*}
  k(t; w_0, m) \defs m \left( \intom w^\chi(\cdot, t; w_0, m) \right)^{-1}, \qquad t \in (0, \tmax(w_0, m)),
\end{align*}
wherein $w_0$ is as in Lemma~\ref{lm:w0}
and $w(\cdot, \cdot; w_0, m)$ and $\tmax(w_0, m)$ denote the solution of \eqref{prob:w} and its maximal existence time, respectively,
given by Lemma~\ref{lm:ex_w}.

Indeed, given the bound $k(t, w_0, m) \ge 2\lambda$,
we are able to conclude finiteness of $\tmax(w_0, m)$ in a quantifiable way in terms of the constants $\lambda$ and $M$
since $w(\cdot, \cdot; w_0, m)$ is then seen to be a supersolution to an equation whose solution already blows up in finite time.
The latter assertion is the content of the following
\begin{lemma}\label{lm:z_blow_up}
  Let $M \gt 0$ and $w_0$ be as in Lemma~\ref{lm:w0}.
  There exists $T_z \equiv T_z(w_0) \in (0, \frac{M^{-\chi}}{\lambda \chi})$
  and a nonnegative $z \in C^0(\Ombar \times [0, T_z)) \cap C^{2, 1}(\Ombar \times (0, T_z))$ solving
  \begin{equation}\label{eq:z_blow_up:z_eq}
  \left\{\; \begin{aligned}
    z_t &= \Delta z + 2\lambda z^{\chi+1} && \text{in $\Omega \times (0, T_z)$}, \\
    \nabla z \cdot \nu &= 0               && \text{on $\partial \Omega \times (0, T_z))$}, \\ 
    z(\cdot, 0) &= w_0                    && \text{in $\Omega$}
  \end{aligned}\right.
  \end{equation}
  classically with the property that $\limsup_{t \nea T_z(w_0)} \|z(\cdot, t)\|_{\leb\infty} = \infty$.
  (We recall that $\lambda$ has been fixed above and is given by Lemma~\ref{lm:w0}.)
\end{lemma}
\begin{proof}
  Existence, uniqueness, regularity and the blow-up criterion of the maximally extended classical solution to (\ref{eq:z_blow_up:z_eq}) can be derived similarly as in the proof of Lemma~\ref{lm:ex_w}. The crucial upper bound $\frac{M^{-\chi}}{\lambda \chi}$ for the existence time $T_z(w_0)$ follows from applying the maximum principle to $w_t - \lambda w^{\chi+1}$ and making use of \eqref{eq:w0:near_bu}; we refer to \cite[Lemma~44.11]{QuittnerSoupletSuperlinearParabolicProblems2019} for details.
\end{proof}

Since the upper bound for $T_z$ given by Lemma~\ref{lm:z_blow_up} converges to $0$ for $M \ra \infty$,
it actually suffices to obtain a lower bound for $k(\cdot; w_0, m)$ only for small times---as long at it is independent of $M$.
This will be  achieved as follows:
Setting $m_1 \defs \frac{4 \lambda}{A}$ and
\begin{align*}
  T_0(w_0, m) \defs \min\{m^{-2}, \tmax(w_0, m)\},
\end{align*}
we see that
\begin{align*}
  T_1(w_0, m) \defs \sup\big\{\, \tau \in [0, T_0(w_0, m)): k(t; w_0, m) \ge 2 \lambda \text{ for all } t \in [0, \tau] \,\big\} \in (0, T_0(w_0, m)],
\end{align*}
is well-defined for all $m \ge m_1$ as then $k(0; w_0, m) \geq m A \ge 4 \lambda \gt 2\lambda$.
We will then show in Lemma~\ref{lm:t1_eq_t0} that $T_1(w_0, m) = T_0(w_0, m)$---%
provided $m$ is sufficiently large, independently of $M$.

The main ingredient of the corresponding proof consists of applying the maximum principle to
$J \defs w_r + \eta r w^q$
in $[0, \tfrac R2] \times [0, T_1(w_0, m))$ for certain $\eta \gt 0$ and $q \gt 1$.
In order to prepare this argument,
the following two lemmata ensure that we are able to control $J(\frac R2, \cdot)$ independently of $m$ and $M$,
each focussing on a summand in the definition of $J$.
\begin{lemma}\label{lm:w_r_bdd}
  There is $C \gt 0$ such that for any $M \ge 0$, $w_0$ as in Lemma~\ref{lm:w0} and $m \gt 0$,
  the solution $w$ of \eqref{prob:w} given by Lemma~\ref{lm:ex_w} satisfies
  $w_r \le 0$ in $(0, R) \times (0, T_0(w_0, m))$ and
  \begin{align*}
    w_r(\tfrac R2, t) \le -C
    \qquad \text{for all $t \in (0, T_0(w_0, m))$}.
  \end{align*}
\end{lemma}
\begin{proof}
  Since
  \begin{align*}
    &\pe  w_{rt} - w_{rrr} + \frac{n-1}{r^2} w_r - \frac{n-1}{r} w_{rr}
     =    \left(w_t - w_{rr} - \frac{n-1}{r} w_r\right)_r \\
    &=    k(t; w_0, m) (w^{\chi+1})_r
     =    (\chi+1) k(t; w_0, m) w^\chi w_r
    \qquad \text{in $(0, R) \times (0, \tmax(w_0, m))$},
  \end{align*}
  the function $z \defs w_r$ solves
  \begin{equation*}
    \left\{\;\begin{aligned}
      z_t - z_{rr} - \tfrac{n-1}{r} z_r + (\tfrac{n-1}{r^2} - (\chi+1) k(t; w_0, m) w^\chi) z &= 0
        && \text{in $(0, R) \times (0, \tmax(w_0, m))$}, \\
      z &= 0
        && \text{on $\{0, R\} \times (0, \tmax(w_0, m))$}, \\
      z(\cdot, 0) &= w_{0r}
        && \text{in $(0, R)$}.
    \end{aligned}\right.
  \end{equation*}
  Thus, the maximum principle asserts $z \le 0$ in $(0, R) \times (0, \tmax(w_0, m))$,
  which implies that $z$ also satisfies
  \begin{equation*}
    \left\{\begin{aligned}
      z_t - z_{rr} - \tfrac{n-1}{r} z_r + \tfrac{n-1}{r^2} z &\le 0 && \text{in $(\tfrac R4, R) \times (0, \tmax(w_0, m))$}, \\
      z &\le 0                                                      && \text{on $\{\tfrac R4, R\} \times (0, \tmax(w_0, m))$}, \\
      z(\cdot, 0) = W_r &\lt 0                                      && \text{in $(\tfrac R4, R)$},
    \end{aligned}\right.
  \end{equation*}
  where $W$ is given by Lemma~\ref{lm:w0}.
  Therefore, the statement is a direct consequence of the strict maximum principle (and finiteness of $T_0(w_0, m)$).
\end{proof}

The following lemma,
where we inter alia obtain upper bounds for $w(\frac R2, \cdot)$,
constitutes a major difference in reasoning compared to \cite[Section~44.2]{QuittnerSoupletSuperlinearParabolicProblems2019}.
There, the mass of the solution to the considered equation can be easily controlled,
directly implying pointwise upper estimates for radially decreasing $w$.

Evidently, this no longer works for the system \eqref{prob:w}.
Of course,
the definition of $T_1(w_0, m)$ already entails a bound
for $\intom w^\chi(\cdot, t)$ and therefore also for $w(\frac R2, t)$ for $t \in (0, T_1(w_0, m))$,
but its dependence on $m$ would not allow us to eventually conclude $T_1(w_0, m) = T_0(w_0, m)$ for sufficiently large $m$.

Thus, we need to make more subtle use of the bounds implied by the definition of $T_1(w_0, m)$.
The main idea is a cut-off argument which allows us to essentially consider \eqref{prob:w} in annuli centered at zero instead of $\Omega$,
so that due to the radially symmetry the problem becomes spatially one-dimensional.
Then Sobolev embedding theorems are strong enough to allow for even $L^\infty$~bounds in these annuli.
\begin{lemma}\label{lm:w_pw_local_bdd}
  There exists $C \gt 0$ such that
  for all $m \ge \max\{m_1, 1\}$, $M \gt 0$ and $w_0$ as constructed in Lemma~\ref{lm:w0},
  the corresponding solution $w$ of \eqref{prob:w} given by Lemma~\ref{lm:ex_w} fulfills
  \begin{align}\label{eq:w_pw_local_bdd:statement}
    w(r, t) \le C
    \qquad \text{for all $r \ge \tfrac R2$ and $t \in (0, T_1(w_0, m))$.}
  \end{align}
\end{lemma}
\begin{proof}
  We fix $\delta \defs \min\{\chi, \frac12\}$,
  $r_0 \defs \frac R4$, $r_1 \defs \frac R3$
  as well as $\varphi \in C^\infty([0, R])$ with $\varphi = 0$ in $[0, r_0]$ and $\varphi = 1$ in $[r_1, R]$.
  By potentially switching to $\varphi^\frac{2}{\delta}$ as a measure to ensure that $\varphi$ to the power $\frac{\delta}{2}$ remains smooth, we may without loss of generality assume that
  there is $K_1 \gt 0$ such that
  \begin{align}\label{eq:w_pw_local_bdd:varphi}
    |\nabla \varphi| \le K_1 \varphi^{1-\frac{\delta}{2}}
    \qquad \text{ in $(0, R)$}.
  \end{align}
  Moreover, as $\frac{\frac12 - \frac1\infty}{\frac12 - \frac12 + \frac11} = \frac12 \in (0, 1)$,
  the Gagliardo--Nirenberg inequality asserts that there is $K_2 \gt 0$ with the property that
  \begin{align*}
          \|\psi\|_{\leb[(0,R)]\infty}^2
    &\le  K_2 \|\psi_r\|_{\leb[(0,R)]2} \|\psi\|_{\leb[(0,R)]2}
          + K_2 \|\psi\|_{\leb[(0,R)]2}^2
    \qquad \text{for all $\psi \in \sob[(0,R)]12$},
  \end{align*}
  which when combined with Young's inequality entails that
  \begin{align}\label{eq:w_pw_local_bdd:ehrling}
        2 m \|\psi\|_{\leb[(0,R)]\infty}^2
    \le \frac{\omega_{n-1}}{r_0^{1-n}} \|\psi_r\|_{\leb[(0,R)]2}^2 + K_3 m^2 \|\psi\|_{\leb[(0,R)]2}^2
  \end{align}
  for all $\psi \in \sob[(0,R)]12$ and $m \ge 1$,
  wherein $K_3 \defs \frac{K_2^2 r_0^{1-n}}{\omega_{n-1}} + 2K_2$.
  We now fix first $m \ge \max\{m_1, 1\}$ as well as $M \gt 0$
  and then $w_0$ as well as $w \defs w(\cdot, \cdot; w_0, m)$ as given by Lemma~\ref{lm:w0} and Lemma~\ref{lm:ex_w}, respectively.
  Testing the first equation in \eqref{prob:w} with $w \varphi^2$ gives
  \begin{align*}
          \frac12 \ddt \intom (w \varphi)^2
          + \intom \nabla w \cdot \nabla (w \varphi^2)
    &=    m \left( \intom w^\chi \right)^{-1} \intom w^{\chi+2} \varphi^2
     \le  m \|w \varphi\|_{\leb\infty}^2
    \qquad \text{in $(0, \tmax(w_0, m))$}.
  \end{align*}
  Since 
  \begin{align*}
        \nabla w \cdot \nabla (w \varphi^2) 
    &=  \varphi \nabla w \cdot \nabla (w \varphi)
        + w \varphi \nabla w \cdot \nabla \varphi \\
    &=  (\varphi \nabla w + w \nabla \varphi) \cdot \nabla (w \varphi)
        + (w \varphi \nabla w - w \nabla (w \varphi)) \cdot \nabla \varphi \\
    &=  |\nabla (w \varphi)|^2
        - w^2 |\nabla \varphi|^2
    \qquad \text{in $\Omega \times (0, \tmax(w_0, m))$},
  \end{align*}
  this implies
  \begin{align}\label{eq:w_pw_local_bdd:ddt_w2}
          \frac12 \ddt \intom w^2 \varphi^2
          + \intom |\nabla (w \varphi)|^2
    &\le  m \|w \varphi\|_{\leb\infty}^2
          + \intom w^2 |\nabla \varphi|^2
    \qquad \text{in $(0, \tmax(w_0, m))$}.
  \end{align}

  Turning our attention to the second term on the right-hand side herein first,
  we note that due to $w \ge 1$ and by the definition of $T_1(w_0, m)$,
  \begin{align*}
        \intom w^\delta(\cdot, t)
    \le \intom w^\chi(\cdot, t)
    =     m k^{-1}(t)
    \le m (2 \lambda)^{-1}
    \qquad \text{for all $t \in (0, T_1(w_0, m))$}.
  \end{align*}
  Combined with \eqref{eq:w_pw_local_bdd:varphi} and Young's inequality, this gives
  \begin{align}\label{eq:w_pw_local_bdd:rhs_1}
          \intom w^2 |\nabla \varphi|^2
    &\le  K_1^2 \intom w^\delta w^{2-\delta} \varphi^{2-\delta} \notag \\
    &\le  K_1^2  (2 \lambda)^{-1} m \|w \varphi\|_{\leb\infty}^{2-\delta} \notag \\
    &\le  m \|w \varphi\|_{\leb\infty}^2 + K_4 m
    \qquad \text{in $(0, T_1(w_0))$}
  \end{align}
  for $K_4 \defs (K_1^2 (2 \lambda)^{-1})^{\frac 2\delta}$.

  Regarding the first term on the right-hand side in both \eqref{eq:w_pw_local_bdd:ddt_w2} and \eqref{eq:w_pw_local_bdd:rhs_1},
  we make use of \eqref{eq:w_pw_local_bdd:ehrling} and the fact that $\varphi \equiv 0$ in $(0, r_0)$ to obtain
  \begin{align}\label{eq:w_pw_local_bdd:rhs_2}
          2 m \|w \varphi\|_{\leb[(0,R)]\infty}^2
    &\le  \frac{\omega_{n-1}}{r_0^{1-n}} \int_0^R (w \varphi)_r^2 + K_3 m^2 \int_0^R (w \varphi)^2 \notag \\
    &\le  \omega_{n-1} \int_0^R r^{n-1} (w \varphi)_r^2 + K_3 r_0^{1-n} m^2 \int_0^R  r^{n-1} (w \varphi)^2 \notag \\
    &=    \intom |\nabla (w \varphi)|^2 + K_5 m^2 \intom (w \varphi)^2
    \qquad \text{in $(0, \tmax(w_0))$},
  \end{align}
  where $K_5 \defs \frac{K_3 r_0^{1-n}}{\omega_{n-1}}$.

  In conjunction, \eqref{eq:w_pw_local_bdd:ddt_w2}, \eqref{eq:w_pw_local_bdd:rhs_1} and \eqref{eq:w_pw_local_bdd:rhs_2} imply
  \begin{align*}
          \ddt \intom (w \varphi)^2
    &\le  2K_5 m^2 \intom (w \varphi)^2 + 2K_4 m
    \qquad \text{in $(0, T_1(w_0, m))$},
  \end{align*}
  whence, by the variation-of-constants formula, \eqref{eq:w0:w0_W}
  and the facts that $\|\varphi\|_{L^\infty(\Omega)} \le 1$ and $T_1(w_0, m) \le T_0(w_0, m) \le m^{-2} \le m^{-1}$,
  \begin{align}\label{eq:w_pw_local_bdd:w_varphi_l2}
          \intom (w(\cdot, t) \varphi)^2
    &\le  \ure^{2 K_5 m^2 t} \left(
            \intom (w_0 \varphi)^2
            + 2 K_4 m \int_0^t \ure^{-2 K_5 m^2 s} \ds
          \right) \notag \\
    &\le  \ure^{2 K_5} \left(
            \|W\|_{L^2(\Omega \setminus \ol B_{r_0}(0)}^2
            + 2 K_4
          \right)
    \sfed K_6
    \qquad \text{for all $t \in (0, T_1(w_0, m))$}.
  \end{align}

  We now claim that this implies \eqref{eq:w_pw_local_bdd:statement} for $C \defs \sqrt\frac{12 \cdot 3^{n-1}K_6}{R^n \omega_{n-1}}$.
  Indeed, suppose that there are $r' \in [\frac R2, R)$ and $t' \in (0, T_1(w_0, m))$ with $w(r', t') \gt C$.
  As $w$ is radially decreasing by Lemma~\ref{lm:w_r_bdd},
  this would in particular imply $w(r, t') \ge C$ for all $r \in [0, \frac R2]$.
  But then
  \begin{align*}
          \intom (w(x, t') \varphi)^2 \dx
    &\ge  \omega_{n-1} \int_{\frac R3}^{\frac R2} \rho^{n-1} w^2(\rho, t') \drho
     \ge  \frac{C^2 R^{n-1} \omega_{n-1}}{3^{n-1}} \int_{\frac R3}^{\frac R2} 1 \drho
     =    \frac{C^2 R^n \omega_{n-1}}{3^{n-1}\cdot 6}
     =    2K_6 
     \gt  K_6,
  \end{align*}
  contradicting \eqref{eq:w_pw_local_bdd:w_varphi_l2}.
\end{proof}

With Lemma~\ref{lm:w_r_bdd} and Lemma~\ref{lm:w_pw_local_bdd} at hand,
we are now able to make use of a comparison principle argument allowing us to estimate $w$ in $\Ombar \times [0, T_1(w_0, m))$,
independently of $m$ and $M$.

\begin{lemma}\label{lm:w_pw_global_bdd}
  Let $q \in (1, \chi+1)$.
  Then there is $C \gt 0$ such that for any $m \ge \max\{m_1, 1\}$ and $M \gt 0$,
  the solution $w$ of \eqref{prob:w} given by Lemma~\ref{lm:ex_w} emanating from the initial data $w_0$ constructed in Lemma~\ref{lm:w0}
  satisfies 
  \begin{align}\label{eq:w_pw_global_bdd:statement}
    w(r, t) \le C r^{-\frac{2}{q-1}}
    \qquad \text{for all $r \in (0, R)$ and $t \in (0, T_1(w_0, m))$}.
  \end{align}
\end{lemma}
\begin{proof}
  The proof is similar to \cite[Lemma~44.12]{QuittnerSoupletSuperlinearParabolicProblems2019}.
  However, due to the importance of this lemma and as some modifications are necessary,
  we choose to nonetheless give a full proof here.

  We denote by $K_1$ and $K_2$ the constants given by Lemma~\ref{lm:w_pw_local_bdd} and Lemma~\ref{lm:w_r_bdd}, respectively,
  and set
  \begin{align}\label{eq:w_pw_global_bdd:def_eta}
    \eta \defs \min\left\{\mu, \frac{2K_2}{R K_1^q}, \frac{(\chi+1-q) \lambda}{q}\right\}.
  \end{align}

  We now fix $m \ge \max\{m_1, 1\}$, $M \gt 0$, $w_0$ as in Lemma~\ref{lm:w0} and the solution $w$ of \eqref{prob:w} given by Lemma~\ref{lm:ex_w}.
  The claim then follows once we have shown that
  \begin{align}\label{eq:w_pw_global_bdd:J}
    J \defs w_r + \eta r w^q \le 0
    \qquad \text{in $(0, \tfrac R2) \times (0, T_1(w_0, m))$}.
  \end{align}
  Indeed, \eqref{eq:w_pw_global_bdd:J} implies $(w^{1-q}(r, t))_r \ge (q-1) \eta r$
  and thus $w(r, t) \le (\frac{1}{2}(q-1) \eta)^{-\frac{1}{q-1}} r^{-\frac{2}{q-1}}$
  for all $(r, t) \in (0, \frac R2) \times (0, T_1(w_0, m))$.
  Since $w$ is radially decreasing, \eqref{eq:w_pw_global_bdd:statement} is
  a consequence of the above for some $C \gt 0$ only depending on $\eta$, $q$ and $R$.
  
  In order to prove \eqref{eq:w_pw_global_bdd:J},
  we first calculate
  \begin{align*}
          J_t - J_{rr}
    &=    (w_t - w_{rr})_r
          + \eta \left(r (w^q)_t - (r w^q)_{rr}\right) \\
    &=    \left(w_t - w_{rr} - \frac{n-1}{r} w_r \right)_r - \frac{n-1}{r^2} w_r
          + \frac{n-1}{r} w_{rr} \\
    &\pe  + \eta \left(q r w^{q-1} (w_t - w_{rr}) - 2q w^{q-1} w_r - q(q-1) r w^{q-2} (w_r)^2\right) \\
    &\le  \left(k(t; w_0, m) (\chi+1) w^\chi - \frac{n-1}{r^2} - 2q \eta w^{q-1} \right) w_r \\
    &\pe  + \frac{n-1}{r} \left(J - \eta r w^q \right)_r
          + q \eta r w^{q-1} \left(\frac{n-1}{r} w_r + k(t; w_0, m) w^{\chi+1} \right) \\
    &=    \left(k(t; w_0, m) (\chi+1) w^\chi - \frac{n-1}{r^2} - 2 q \eta w^{q-1} \right) \left( J - \eta r w^q \right) \\
    &\pe  + \frac{n-1}{r} J_r - \frac{n-1}{r} \eta w^q
          + q \eta r k(t; w_0, m) w^{q+\chi} \\ 
    &= a(r, t) J + \frac{n-1}{r} J_r + b(r, t)
    \qquad \text{in $(0, R) \times (0, \tmax(w_0, m))$},
  \end{align*}
  where
  \begin{align*}
    a(r, t) &\defs k(t; w_0, m) (\chi+1) w^\chi(r, t) -\frac{n-1}{r^2} - 2q \eta w^{q-1}(r, t)
  \intertext{and}
    b(r, t) &\defs \eta r w^{q+\chi}(r, t) [ 2q \eta w^{-(\chi+1-q)}(r, t) - (\chi+1-q) k(t; w_0, m) ]
  \end{align*}
  for $(r, t) \in (0, R) \times (0, \tmax(w_0, m))$.
  By the definitions of $T_1(w_0, m)$ and $k(\cdot; w_0, m)$, as $w \ge 1$ and because of $\eta \le \frac{(\chi+1-q)\lambda}{q}$,
  we can estimate
  \begin{align*}
          b(r, t)
    &\le  2\eta r w^{q+\chi} [q \eta - (\chi+1-q) \lambda ]
    \le   0
    \qquad \text{for all $(r, t) \in (0, R) \times (0, T_1(w_0, m))$}.
  \end{align*}
  As moreover $J(\cdot, 0) = w_{0r} + \eta r w_0^q \le w_{0 r} + \mu r w_0^{\chi + 1} \le 0$ in $(0, \frac R2)$ due to \eqref{eq:w_pw_global_bdd:def_eta}, nonnegativity of $w-1$ and \eqref{eq:w0:comp}
  and $J(0, \cdot) = w_r(0, \cdot) \le 0$ by Lemma~\ref{lm:w_r_bdd} as well as
  $J(\frac R2, \cdot) \le -K_2 + \eta \frac R2 K_1^q \le 0$ by Lemma~\ref{lm:w_r_bdd} and Lemma~\ref{lm:w_pw_local_bdd}
  in $(0, T_1(w_0, m))$,
  the maximum principle asserts that indeed \eqref{eq:w_pw_global_bdd:J} holds.
\end{proof}

Next, we show that for sufficiently large $m$, Lemma~\ref{lm:w_pw_global_bdd} entails $T_1(w_0, m) = T_0(w_0, m)$, which means that the problematic growth dampening term in (\ref{prob:w}) is in fact uniformly bounded on the whole time interval $(0,T_0(w_0, m))$ in a fashion independent of $M$.

\begin{lemma}\label{lm:t1_eq_t0}
  There is $m_0 \ge \max\{m_1, 1\}$ such that for all $m \ge m_0$, $M \gt 0$ and $w_0$ as given by Lemma~\ref{lm:w0},
  we have $T_1(w_0, m) = T_0(w_0, m)$. 
\end{lemma}
\begin{proof}
  Since $n \ge 3$, we may fix $q \in (\frac{2 \chi}{n} + 1, \chi + 1)$.
  According to Lemma~\ref{lm:w_pw_global_bdd}, there is $K_1 \gt 0$ not depending on $m$ and $M$ such that
  the solution $w \defs w(\cdot, \cdot; w_0, m)$ given by Lemma~\ref{lm:ex_w} for $w_0$ as in Lemma~\ref{lm:w0} fulfills
  \begin{align*}
        \intom w^\chi(\cdot, t)
    \le K_1^\chi \omega_{n-1} \int_0^R \rho^{n-1} \rho^{-\frac{2 \chi}{q-1}} \drho
    \sfed K_2
    \qquad \text{in $(0, T_1(w_0, m))$}.
  \end{align*}
  As $n-1 - \frac{2 \chi}{q-1} \gt -1$, $K_2$ is finite so that $m_0 \defs \max\{4 K_2 \lambda, m_1, 1\}$ is finite (and independent of $m$ and $M$) as well.
  Thus, $k(t; w_0, m) \ge K_2^{-1} m \ge 4\lambda$ in $(0, T_1(w_0, m))$ provided $m \ge m_0$,
  that is, $T_1(w_0, m) = T_0(w_0, m)$ for these $m$ due to the definition of $T_1(w_0, m)$.
\end{proof}

Finally using the above uniform (in terms of $M$) bound on the dampening influence $k(t; w_0, m)$ in our considered problem, we will now close the central argument of this section by showing that as we increase $M$ in our initial data construction, the associated solution to (\ref{prob:w}) blows up earlier than any given time $T > 0$. As already alluded to before, this is done by comparison with the system discussed in Lemma~\ref{lm:z_blow_up}. 
\begin{proof}[Proof of Lemma~\ref{lm:w_arbitrarily_earily_blowup}] 
	Let $m_0$ be as in Lemma~\ref{lm:t1_eq_t0} and $m \ge m_0$. We then fix some $T > 0$, without loss of generality assuming that $T < m^{-2}$, and initial data $w_0$ as constructed in Lemma~\ref{lm:w0} with $M > (\lambda \chi T)^{-\frac{1}{\chi}}$.

	According to the definition of $T_1(w_0, m)$ and due to Lemma~\ref{lm:t1_eq_t0},
	the unique solution $w \defs w(\cdot, \cdot; w_0, m)$ of \eqref{prob:w} given by Lemma~\ref{lm:ex_w} with initial data $w_0$ 
	is a supersolution of \eqref{eq:z_blow_up:z_eq} in $\Ombar \times [0, T_0(w_0, m))$.
 
	Let $T_z(w_0)$ and $z$ be as given by Lemma~\ref{lm:z_blow_up}.
	For the sake of contradiction, we suppose that $\tmax(w_0, m) \ge T$.
	Then $T_0(w_0, m) \ge T > \frac{M^{-\chi}}{\lambda \chi} \ge T_z$
	so that the comparison principle would assert $w \ge z$ in $\Ombar \times [0, T_z)$,
	which due to Lemma~\ref{lm:z_blow_up} would imply
	$\limsup_{t \nea T_z} \|w(\cdot, t)\|_{\leb\infty} \ge \limsup_{t \nea T_z} \|z(\cdot, t)\|_{\leb\infty} = \infty$
	and hence $\tmax(w_0, m) \le T_z(w_0) < T$.
\end{proof}

\section{Proof of Theorem~\ref{th:main}}
We have now proven in Section~\ref{sec:reduction} that, if the (unique) maximally extended solutions to (\ref{problem}) do not exhibit the unboundedness property described in (\ref{eq:th:main}) on a time interval $(0,T)$, we can construct a classical solution to (\ref{prob:w}) on the same time interval with initial data $v_0$ and parameter $m = \int_\Omega u_0$ and in Section~\ref{sec:limit_blowup} that we can in fact construct initial data such that the (unique) solution to (\ref{prob:w}) blows up arbitrarily early given that the parameter $m$ is sufficiently large. Combining these two insight then yields the following straightforward proof by contradiction for the central result of this paper:

\begin{proof}[Proof of Theorem~\ref{th:main}]
	We begin by fixing $m_0$ as in Lemma~\ref{lm:w_arbitrarily_earily_blowup}. We further fix $T > 0$ and initial data $w_0 \sfed v_0 \in C^2(\overline{\Omega})$ such that the associated solution to (\ref{prob:w}) blows up before time $T$, again according to Lemma~\ref{lm:w_arbitrarily_earily_blowup}. We further fix some nonnegative initial data $u_0 \in C^0(\Omega)$ with mass $\int_\Omega u_0 \sfed m > m_0$. We can then use Lemma~\ref{lm:local_exist} to construct the unique maximally extended classical solutions $(\ue, \ve)$ to (\ref{problem}) associated with the above initial data for each $\eps \in (0,1)$.

	To facilitate a proof by contradiction, we now assume that the solutions fixed above do not have the property (\ref{eq:th:main}) from Theorem~\ref{th:main}. Given the blow-up criterion (\ref{eq:blowup_criterion}), this means that there must exist $\eps^\star \in (0,1)$ such that, for all $\eps \in (0,\eps^\star)$, the solutions $(\ue, \ve)$ must exist on the time interval $(0,T)$ as well as be uniformly bounded in some $L^p(\Omega)$ with $p > \frac{n}{2}$ on said time interval. But this is, of course, exactly the property needed to apply Lemma~\ref{lm:limit_reduction} to construct a classical solution to (\ref{prob:w}) on $(0,T)$ with initial data $w_0 = v_0$. Given that such a solution is unique according to Lemma~\ref{lm:ex_w} and, by our choice of $v_0$ and $m_0$, we further know that the constructed solution must in fact blow up at some time before $T$, this leads to the anticipated contradiction. As such, the solutions $(\ve, \ue)$ fixed above must have already satisfied the property (\ref{eq:th:main}), completing the proof.
\end{proof}

\section*{Acknowledgments}
\small The first author has been partly supported by the German Academic Scholarship Foundation.
Both authors acknowledge support by the Deutsche Forschungsgemeinschaft within the project \emph{Emergence of structures and advantages in
cross-diffusion systems}, project number 411007140.


\end{document}